\documentclass{amsart}
\usepackage{amsthm,amsmath,amssymb,amsfonts,stmaryrd}
  \newtheorem{thm}{Theorem}[section]
  \newtheorem{lem}[thm]{Lemma}
  \newtheorem{prop}[thm]{Proposition}
  \newtheorem{cor}[thm]{Corollary}
  \theoremstyle{definition}
  \newtheorem{defn}[thm]{Definition}
  \newtheorem{exm}[thm]{Example}
  \newtheorem{rmk}[thm]{Remark}
  
 \newcommand\ra{\rightarrow}

 \newcommand\s{\subseteq}

\newcommand{\Max}{\mbox{\rm Max}}
\newcommand{\Spec}{\mbox{\rm Spec}}
\newcommand{\Ker}{\mbox{\rm Ker}}
\newcommand{\Rad}{\mbox{\rm Rad}}
\newcommand{\Id}{\mbox{\rm Id}}
\newcommand{\Con}{\mbox{\rm Con}}
\newcommand{\MaxS}{\mbox{\rm MaxS}}
 \numberwithin{equation}{section}

\begin{document}

\title{State BCK-algebras and State-Morphism BCK-algebras}
\author[R. A. Borzooei, A. Dvure\v{c}enskij, O.Zahiri]{R. A. Borzooei$^{1}$, A. Dvure\v{c}enskij$^{2,3}$, and O. Zahiri$^{1}$}
\date{}
\maketitle

\begin{center}  \footnote{


The paper has been supported by Slovak Research and Development Agency under the contract APVV-0178-11, the grant VEGA No. 2/0059/12 SAV and by
CZ.1.07/2.3.00/20.0051. }
Department of Mathematics, Shahid Beheshti University, G. C., Tehran, Iran\\
$^2$ Mathematical Institute,  Slovak Academy of Sciences,\\
\v Stef\'anikova 49, SK-814 73 Bratislava, Slovakia\\
$^3$ Depart. Algebra  Geom.,  Palack\'{y} Univer.\\
17. listopadu 12,
CZ-771 46 Olomouc, Czech Republic\\
E-mail: {\tt borzooei@sbu.ac.ir} \quad {\tt dvurecen@mat.savba.sk}\quad   {\tt om.zahiri@gmail.com}
\end{center}

\begin{center}
\small{\it Dedicated  to Prof. J\'an Jakub\'\i k on the occasion of his $90^{th}$ birthday}
\end{center}

\begin{abstract}
In the paper, we define the notion of a state BCK-algebra and a state-morphism BCK-algebra extending the language of BCK-algebras by adding a unary operator which models probabilistic reasoning.
We present a relation between state  operators and state-morphism operators and  measures and states on BCK-algebras, respectively. We study subdirectly irreducible state (morphism) BCK-algebras.
We introduce the concept of an adjoint pair in BCK-algebras and show that there is a one-to-one correspondence between adjoint pairs and state-morphism operators.
In addition, we show the generators of  quasivarieties of state-morphism BCK-algebras.
\end{abstract}

\section{ Introduction}

In 1966, Imai and Iseki \cite{2,3} introduced two classes of
abstract algebras: {\it BCK-algebras} and {\it BCI-algebras}. These
algebras have been intensively studied by many authors. For  a comprehensive overview on BCK-algebras, we recommend the book \cite{BCK}.
It is known that the class of BCK-algebras is a proper subclass
of the class of BCI-algebras. MV-algebras were introduced by Chang in \cite{MV}, in order to show that
\L ukasiewicz logic is complete with respect to evaluations of propositional variables
in the real unit interval $[0,1]$. It is well known that
the class of MV-algebras is a proper subclass of the class of BCK-algebras. Therefore, both BCK-algebras and MV-algebras are important for the study of fuzzy logic.

In \cite{Mun3}, Mundici introduced a {\it state} on MV-algebras as averaging  the truth value in \L ukasiewicz logic. States constitute measures on their associated MV-algebras
which generalize the usual probability measures on Boolean algebras.
Kroupa \cite{12} and Panti \cite{Pan} have recently shown that every state on an MV-algebra can be presented
as a usual Lebesgue integral over an appropriate space. K\"{u}hr and Mundici \cite{13} studied states via de Finetti's notion of a
coherent state with motivation in Dutch book making. Their method is applicable to other structures besides MV-algebras. Measures on pseudo BCK-algebras were studied in \cite{CD}.

Since MV-algebras with  state are not universal algebras, they do not automatically induce an assertional logic. Recently, Flaminio and Montagna
in \cite{Flamino1, Flamino2} presented an algebraizable logic using a probabilistic
approach, and its equivalent algebraic semantics is precisely the
variety of state MV-algebras. We recall that a {\it state MV-algebra}
is an MV-algebra whose language is extended by adding an operator,
$\mu$ (also called an {\it internal state}), whose properties are
inspired by ones of states. Analogues of extremal states are {\it state-morphism operators}, introduced in \cite{DD1,DD2}, where by definition, a state-morphism is an
idempotent endomorphism on an MV-algebra.

State MV-algebras generalize, for example, H\'ajek's approach,
\cite{Haj}, to fuzzy logic with modality Pr (interpreted as {\it
probably}) which has the following semantic interpretation: The
probability of an event $a$ is presented as the truth value of
Pr$(a)$. On the other hand, if $s$ is a state, then $s(a)$ is
interpreted as the average appearance of the many valued event $a.$

In \cite{Flamino1, Flamino2}, the authors found a relation between states on MV-algebras and state MV-algebras.
In \cite{DD1,DD2}, some results about characterizations of subdirectly  irreducible state-morphism MV-algebras, simple, semisimple, and  local state MV-algebras were shown. In \cite{DDL}, the authors study the variety of state-morphism
MV-algebras together with a characterization of subdirectly irreducible state MV-algebras, and some interesting characterizations of some varieties of state-morphism MV-algebras were given. These results were generalized in \cite{DRS, DKM, BoDv}.

In the present paper, we concentrate to the study of state BCK-algebras and state-morphism BCK-algebras. We show their basic properties and we characterize quasivarieties of state-morphism BCK-algebras and their generators. We present that the generator of a quasivariety of state-morphism BCK-algebras consists of diagonal state-morphism BCK-algebras. The goal of the present paper is to extend the study of state MV-algebras to state BCK-algebras. We note that in contrast to MV-algebras, in this case we have to deal  with quasivarieties because the class of BCK-algebras forms a quasivariety and not a variety.

We note that a state-morphism BCK-algebra is a special case of algebras with a distinguished idempotent endomorphism and such algebras
are not new: experts working in various areas (ranging from computer
science, Baxter algebras, set theory, category theory and homotopy
theory, see e.g. \cite{Vic, AgMo, Sch})
have considered such structures with a fixed endomorphism.

The paper is organized as follows. Section 2 gathers the elements  of BCK-algebras. In Section 3, we introduce the concept of a state
BCK-algebra and we study its properties. Then we verify a subdirectly irreducible state BCK-algebra and we
characterize this structure. We show that if $X$ is a bounded commutative BCK-algebra, then
$(X,\mu)$ is a state (morphism) MV-algebra if and only if $(X,\mu)$ is a state (morphism) BCK-algebra such that $\mu(1)=1$.
In Section 4, we study state-morphism BCK-algebras  and state ideals. Some relations between congruence relations
on state-morphism BCK-algebras and state ideals are also obtained. Then we introduce the concept of an adjoint pair in a BCK-algebra and describe a relation between state-morphism operators and adjoint pairs in BCK-algebras. Finally, Section 5 gives results on generators of quasivarieties of state-morphism BCK-algebras, and we present two open problems.

\section{Preliminaries}

In the section, we gather some basic notions relevant to BCK-algebras and
MV-algebras which will need in the next sections.

We say that an MV-algebra is an algebra $(M,\oplus,',0)$ of type $(2,1,0)$, where $(M,\oplus,0)$ is a
commutative monoid with neutral element $0$ and for all $x,y\in M$:

\begin{enumerate}
\item[(i)] $x''=x$;

\item[(ii)] $x\oplus 1=1$, where $1=0'$;

\item[(iii)] $x\oplus (x\oplus y')'=y\oplus (y\oplus x')'$.
\end{enumerate}

In any MV-algebra $(M,\oplus,',0)$, we can define the following further operations:
$$x\odot y=(x'\oplus y')',\quad x\ominus y=(x'\oplus y)'.
$$

A {\it state MV-algebra} is a pair $(M,\sigma)$ such that $(M,\oplus,',0)$ is an MV-algebra and $\sigma$ is a unary operation on $M$ satisfying:

\begin{enumerate}
\item[(1)] $\sigma(1)=1$;

\item[(2)] $\sigma(x')=\sigma(x)'$;

\item[(3)] $\sigma(x\oplus y)=\sigma(x)\oplus \sigma(y\ominus (x\odot y))$;

\item[(4)] $\sigma(\sigma(x)\oplus\sigma(y))=\sigma(x)\oplus\sigma(y)$.
\end{enumerate}

In \cite{DD1}, Di Nola and Dvure\v{c}enskij have introduced a {\it state-morphism operator}
on an MV-algebra $(M,\oplus,',0)$ as an MV-homomorphism
$\sigma:M\ra M$ such that $\sigma^2=\sigma$ and the pair $(M,\sigma)$ is said to be a {\it state-morphism} MV-algebra. They have proved that
the class of state-morphism MV-algebras is a proper subclass of state MV-algebras.

\begin{defn}\cite{2,3} \label{2.1}
A {\it BCK-algebra} is an algebra $(X,*,0)$ of type
$(2,0)$ satisfying the following conditions:

\begin{enumerate}
\item[(BCK1)] $((x*y)*(x*z))*(z*y)=0$;

\item[(BCK2)] $x*0=x$;

\item[(BCK3)] $x*y=0$ and $y*x=0$ imply $y=x$;

\item[(BCK4)] $0*x=0$.
\end{enumerate}
\end{defn}

A BCK-algebra $X$ is called {\it non-trivial} if $X\neq\{0\}$.
If $X$ is a BCK-algebra, then the relation $\leq$
defined by $x\leq y \Leftrightarrow x*y=0$,  $x,y\in
X$, is a partial order on $X$. In addition, for all $x,y,z\in X,$ the following
hold:
\begin{enumerate}
\item[(BCK5)] $x*x=0$;

\item[(BCK6)] $(x*y)*z=(x*z)*y$;

\item[(BCK7)] $x\leq y$ implies $x*z\leq y*z$ and $z*y\leq z*x$;

\item[(BCK8)] $x*(x*(x*y))=x*y$;

\item[(BCK9)] $(x*y)*(x*z)\leq z*y$ and $(y*x)*(z*x)\leq y*z$.
\end{enumerate}

In a BCK-algebra $X$, we define $x*y^0=x$ and $x*y^{n}=(x*y^{n-1})*y$
for any integer  $n\ge 1$ and all $x,y \in X.$ A BCK-algebra $(X,*,0)$ is
called {\it bounded} if $(X,\leq)$ has the greatest element, where $\leq$ is
the above defined partially order relation. Let use
denote by $1$ the greatest element of $X$ (if it exits). In bounded BCK-algebras, we usually write $Nx$ instead of $1*x$.
A BCK-algebra $(X,*,0)$ is called a {\it commutative} BCK-algebra
if $x*(x*y)=y*(y*x)$ for all $x,y\in X$. Each commutative BCK-algebra is a lover semilattice and $x\wedge y=x*(x*y)$
for all $x,y\in X$ (see \cite{BCK}). Let $(X,*,0)$ and $(Y,*,0)$ be two BCK-algebras.
A map $f:X\ra Y$ is called a {\it homomorphism}
if $f(a*b)=f(a)*f(b)$ for all $a,b\in X$. Then $f(0)=0$
(since $f(0)=f(0*0)=f(0)*f(0)=0$).

A non-empty subset $I$ of a BCK-algebra $X$ is
called an {\it ideal} if (1) $0\in I$, (2) $y*x\in I$ and
$x\in I$ imply that $y\in I$ for all $x,y\in X$.
We denote by ${\rm I}(X)$, the set of all ideals of $X$. An ideal $I$ of a BCK-algebra $X$ is
called {\it proper} if $I\neq X$. Suppose that
$(X,*,0)$ and $(Y,*,0)$ are two BCK-algebras and $f:X\ra Y$ is a homomorphism,
then $\Ker(f)=f^{-1}(\{0\})$ is an ideal of $X$.
Let use denote by $\langle S\rangle$ the least ideal of $X$ containing $S$,
where $S$ is a subset of a BCK-algebra $X.$ It is called the ideal
generated by $S$. If $S$ is a subset of more BCK-algebras, we will use $\langle S \rangle_X$ to specify a concrete BCK-algebra $X.$ Instead of $\langle \{a\}\rangle$ we will write rather $\langle a \rangle,$ where $a \in X.$

\begin{thm}\cite{chang} \label{2.2}
Let $S$ be a subset of a BCK-algebra $(X,*,0).$ Then
$$\langle S\rangle =\{x\in
X \mid(\cdots((x*a_{1})*a_{2})*\cdots)*a_{n}=0 \mbox{ for some $n\in
\mathbb{N}$ and some $a_{1},\ldots,a_{n}\in S\cup \{0\}$}\}.
$$

Moreover, if $I$ is an ideal of $X$, then
$$\langle I\cup S \rangle=\{x\in
X \mid (\cdots((x*a_{1})*a_{2})*\cdots)*a_{n}\in I  \mbox{ for some $n\in
\mathbb{N}$ and some $a_{1},\ldots,a_{n}\in S$}\}.$$
\end{thm}

Let $I$ be an ideal of a BCK-algebra $(X,*,0)$. Then the relation
$\theta_{I},$ defined by $(x,y)\in \theta_{I}$ if and only if $x*y,y*x\in I,$ is a congruence relation on $X$.
Let us denote by $x/I$ or $[x]$ the set $\{y\in X \mid (x,y)\in\theta_{I}\}$ for all $x\in X$. Then  $(X/I,*,0/I)$
is a BCK-algebra, when $X/I:=\{x/I \mid x\in X\}$ and $x/I*y/I:=(x*y)/I$ for all $x,y\in X$ (see \cite{BCK}).

An ideal $I$ of a BCK-algebra $(X,*,0)$ is called {\it commutative} if $x*y\in I$ implies that
$x*(y*(y*x))\in I$ for all $x,y\in X$. If $I$ is a commutative ideal, the BCK-algebra $X/I$ is a
commutative BCK-algebra \cite[Thm 2.5.6]{chang}.

\begin{thm}\label{con BCK}
Let $(X,*,0)$ be a BCK-algebra and $\theta$ be a congruence relation on $X$.
Then $[0]_\theta$ is an ideal of $X$. Moreover, if $I=[0]_\theta$, then $\theta_{I}=\theta$.
\end{thm}

\begin{proof}
See \cite[Prop 1.5.9, Prop. 1.5.11, Cor. 1.5.12]{chang}.
\end{proof}

\begin{defn}\cite{BZ,chang}
Let $I$ be a proper ideal of a BCK-algebra $(X,*,0)$. Then $I$ is called a
\begin{itemize}
\item {\it prime} ideal if $\langle x\rangle \cap \langle y\rangle\s I$ implies $x\in I$ or $y\in I$ for all $x,y\in X$;
\item {\it maximal} ideal if $\langle I\cup \{x\}\rangle=X$ for all $x\in X-I$.
\end{itemize}

We use $\Max(X)$ and $\Spec(X)$ to denote the set of all maximal and prime ideals of $X$,
respectively. In each BCK-algebra $X$, $\Max(X)\s \Spec(X)$ (see \cite[Thm 3.7]{BZ}).
A BCK-algebra $(X,*,0)$ is called {\it simple} if it has only two ideals and it is
called {\it semisimple} if $\Rad(X):=\bigcap \Max(X)=\{0\}$.
\end{defn}

\begin{defn}\cite{chang}
A BCK-algebra $(X,*,0)$ is {\it positive implicative} if $(x*y)*z=(x*z)*(y*z)$ for all $x,y,z\in X$.
\end{defn}

If $X=[0,a)$ or $X=[0,a]$, where $a\in \mathbb{R}$, or $X =[0,\infty)$, we define the binary operation $*_\mathbb R$ on
$X$ by $x*_{\mathbb{R}}y=\max\{0,x-y\}$.
Then $(X,*_{\mathbb{R}},0)$ is a commutative BCK-algebra (see \cite{BCK}).

\begin{defn}\label{2.3}\cite{Dvu1}
Let $(X,*,0)$ be a BCK-algebra and $m:X\ra [0,\infty]$ be a map such that, for all $x,y\in [0,1]$,
\begin{itemize}
  \item[(i)] if $m(x*y)=m(x)-m(y)$, whenever $y\leq x,$ then $m$ is said to be a {\it measure};
  \item[(ii)] if $1\in X$ and $m$ is a measure with $m(1)=1$, then $m$ is said to be a state;
  \item[(iii)] if $m(x*y)=\max\{0,m(x)-m(y)\}$, then $m$ is said to be a {\it measure-morphism};
  \item[(iv)] if $1\in X$ and $m$ is a measure-morphism with $m(1)=1$, then $m$ is said to be a {\it state-morphism}.
 \end{itemize}
\end{defn}

\section{\bf State BCK-algebras}

In the section, the concept of left and right state BCK-algebras is defined as a
generalization of state MV-algebras, and its properties are studied. We introduce state
ideals and congruence relations of right or left state BCK-algebras, and relations between them are obtained. Finally,
we characterize subdirectly irreducible state BCK-algebras.

From now on, in this paper, $(X,*,0)$ or simply $X$ is a BCK-algebra, unless otherwise specified.

\begin{defn}\label{3.1}
A map $\mu: X\ra X$ is called a {\it left (right) state operator } on $X$ if it satisfies the following conditions:
 \begin{itemize}
 \item[(S0)] $x*y=0$ implies $\mu(x)*\mu(y)=0;$
 \item[(S1)] $\mu(x*y)=\mu(x)*\mu(x*(x*y))$ \quad ($\mu(x*y)=\mu(x)*\mu(y*(y*x))$);
  \item[(S2)] $\mu(\mu(x)*\mu(y))=\mu(x)*\mu(y)$.
 \end{itemize}
A {\it left (right) state BCK-algebra} is a pair $(X,\mu)$,
where $X$ is a BCK-algebra and $\mu$ is a left (right) state operator on $X$.
\end{defn}

Clearly, if $X$ is a commutative BCK-algebra, then $\mu$ is a right state operator on $X$ if and only if it is a left state operator.
In the next proposition, we describe the basic properties of left (right) state operators.

\begin{prop}\label{3.2}
Let $(X,\mu)$ be a left (right) state BCK-algebra. Then, for any $x,y,x_1,\ldots,x_n\in X$,
\begin{itemize}
  \item[{\rm (i)}] $\mu(0)=0$ and $\mu(\mu(x))=\mu(x).$
  \item[{\rm (ii)}] $\mu(x)*\mu(y)\leq \mu(x*y)$. More generally,
  $$
  (\cdots((\mu(x)*\mu(x_1))*\mu(x_2))*\cdots)*\mu(x_n)\leq \mu((\cdots((x*x_1)*x_2)*\cdots)*x_n).
  $$
  \item[{\rm (iii)}] $\Ker(\mu):=\mu^{-1}(\{0\})$ is an ideal of $X.$
  \item[{\rm (iv)}] $\mu(X):=\{\mu(x) \mid x \in X\}$ is a subalgebra of $X$.
\item[{\rm (v)}] $\Ker(\mu)\cap {\rm Im}(\mu)=\{0\}.$
 \end{itemize}
\end{prop}

\begin{proof}
We prove this theorem only for a left state BCK-algebra. The proof for a right state BCK-algebra is similar.

(i) By (BCK4) and (BCK8), we have $\mu(0)=\mu(0*0)=\mu(0)*\mu(0*(0*0))=\mu(0)*\mu(0)=0$. Moreover, by (S2) and (BCK2), we have
$\mu(\mu(x))=\mu(\mu(x)*0)=\mu(\mu(x)*\mu(0))=\mu(x)*\mu(0)=\mu(x)$.

(ii) Let $x,y\in X$. Since $x*(x*y)\leq y$, then $\mu(x*(x*y))\leq \mu(y)$, and so by (BCK7), we get that $\mu(x)*\mu(y)\leq \mu(x)*\mu(x*(x*y))=\mu (x*y)$. The proof of the second part follows from (BCK7).

(iii) By (i), $0\in \Ker(\mu)$. Let $y*x,x\in \Ker(\mu),$ where $x,y\in X$.
Then $\mu(x)=\mu(y*x)=0$. It follows from (ii) that $\mu(y)=\mu(y)*0=\mu(y)*\mu(x)\leq\mu(y*x)=0$,  hence $y\in \Ker(\mu)$. Thus, $\Ker(\mu)$ is an ideal of $X$.

(iv) By (i), $0\in \mu(X)$. Let $a,b\in X$. Then by (S2), $\mu(\mu(a)*\mu(a))=\mu(a)*\mu(b)$ and so $\mu(a)*\mu(b)\in\mu(X)$. Therefore, $\mu(X)$ is a subalgebra of $X$.

(v) It is evident.
\end{proof}

In Theorem \ref{3.3}, we attempt to find a relation between measures and states on BCK-algebras and state BCK-algebras.

\begin{thm}\label{3.3}
Let $a\in[0,1]$, $X=([0,a),*_{\mathbb{R}},0)$ and $(X,\mu)$ be a left state BCK-algebra. Then $\mu:X\ra [0,1]$ is a measure.

In addition, if $X=([0,1],*_{\mathbb{R}},0)$ and
$(X,\mu)$ is a left state BCK-algebra such that $\mu(1)=1$, then $\mu:X\ra [0,1]$ is a state-morphism.
\end{thm}

\begin{proof}
Let $x,y\in X$ such that $y\leq x$. For simplicity, we will write $*=*_\mathbb R.$ Then $\mu(x*y)=\mu(x)*\mu(x*(x*y))$.
Since $X=([0,a),*_{\mathbb{R}},0)$ is a commutative BCK-algebra, then
$x*(x*y)=y*(y*x)=y*0=y$ and so $\mu(x*y)=\mu(x)*\mu(x*(x*y))=\mu(x)*\mu(y)$.
Therefore, $\mu:X\ra [0,1]$ is a measure.

Now, assume that
$X=([0,1],*_{\mathbb{R}},0)$ and $(X,\mu)$ is a left state BCK-algebra. Let $x,y\in X$. Then $\mu(x*y)=\mu(x)*\mu(x*(x*y))$. Since $X$ is linearly ordered, we have two cases. If $x\leq y$, then $\mu(x*y)=\mu(0)=0$ and by Proposition \ref{3.2}(iv), $\mu(x)*\mu(y)=0$ and so $\mu(x*y)=\mu(x)*\mu(y)$.
If $y\leq x$, then $x*(x*y)=y$ (since $([0,1],*_{\mathbb{R}},0)$ is a
commutative BCK-algebra) and so $\mu(x*y)=\mu(x)*\mu(y)$.
Therefore, $\mu:X\ra [0,1]$ is a state-morphism.
\end{proof}

\begin{prop}\label{3.3.1}
Let $(X,\mu)$ be a right state BCK-algebra. Then
 \begin{itemize}
 \item[{\rm (i)}] $y\leq x$ implies $\mu(x*y)=\mu(x)*\mu(y)$ for all $x,y\in X$.
 \item[{\rm (ii)}] $\mu^{-1}(\{0\})$ is a commutative ideal of $X$. Moreover, the map
 $\overline{\mu}:X/\Ker(\mu)\ra X/\Ker(\mu)$  defined by
 $\overline{\mu}(x/\Ker(\mu))=\mu(x)/\Ker(\mu)$ is both a right and left state operator
 on $X/\Ker(\mu)$.
 \item[{\rm (iii)}] $(X,\mu)$ is a left state BCK-algebra.
 \end{itemize}
\end{prop}

\begin{proof}
(i) Let $x,y\in X$ such that $y\leq x$. Then $\mu(x*y)=\mu(x)*\mu(y*(y*x))=\mu(x)*\mu(y*0)=\mu(x)*\mu(y)$.

(ii) By Proposition \ref{3.3}(i), $0\in \mu^{-1}(\{0\})$. Let $x,y*x\in \mu^{-1}(\{0\})$.
Then $\mu(x)=\mu(y*x)=0$ and so
$\mu(y)*\mu(x*(x*y))=0$. Since $\mu(x*(x*y))\leq \mu(x)=0$, then $\mu(y)=0$.
Hence, $\mu^{-1}(\{0\})$ is an ideal of $X$.
Now, let $x*y\in \mu^{-1}(\{0\}).$ 
Since $y*(y*x)\le x,$ by (i), we have
$$ 0= \mu(x*y)= \mu(x)*\mu(y*(y*x))= \mu(x*(y*(y*x))),
$$
which concludes that $x*(y*(y*x))\in \mu^{-1}(\{0\})$. Thus, $\mu^{-1}(\{0\})$ is a commutative ideal of $X$.

It is easy to show that
$\overline{\mu},$ defined by $\overline{\mu}(x/\Ker(\mu)):=\mu(x)/\Ker(\mu),$ ($x \in X$),  is a right state operator on $X/\Ker(\mu).$ In fact, if $x/\Ker(\mu)=y/\Ker(\mu)$,
then $x*y,y*x\in \Ker(\mu)$ and so $\mu(x*y)=\mu(y*x)=0$. Hence
by Proposition \ref{3.2}(ii), $\mu(x)*\mu(y)=\mu(y)*\mu(x)=0$ and so $\mu(x)=\mu(y)$.
Thus, $\overline{\mu}$ is well defined. Since $\Ker(\mu)$ is a
commutative ideal of $X$, then $X/\Ker(\mu)$ is a commutative BCK-algebra, hence
$\overline{\mu}$ is also a left state operator on $X/\Ker(\mu)$.

(iii) Let $x,y\in X$. By (ii), $\Ker(\mu)$ is a commutative ideal of $X$ and so by \cite[Thm 2.5.6]{chang},
$X/\Ker(\mu)$ is a commutative BCK-algebra. Hence, $(x*(x*y))/\Ker(\mu)=(y*(y*x))/\Ker(\mu)$.
Similarly to the proof of (ii), we obtain that $\mu(x*(x*y))=\mu(y*(y*x))$ and so
$\mu(x)*\mu(x*(x*y))=\mu(x)*\mu(y*(y*x))=\mu(x*y)$. Therefore, $(X,\mu)$ is a left state BCK-algebra.
\end{proof}

\begin{cor}
Let $\mu:X\ra X$ be a map. Then $(X,\mu)$ is a right state BCK-algebra if and only if $(X,\mu)$ is a left state
BCK-algebra and $\Ker(\mu)$ is a commutative ideal of $X$.
\end{cor}

\begin{proof}
Suppose that $(X,\mu)$ is a right state BCK-algebra. Then by Proposition \ref{3.3.1}, $(X,\mu)$ is a left state and
$\Ker(\mu)$ is a commutative ideal. Conversely, let $(X,\mu)$ be a left state
BCK-algebra and let $\Ker(\mu)$ be a commutative ideal of $X$. Then for all $x,y\in X$, $(x*(x*y))/\Ker(\mu)=(y*(y*x))/\Ker(\mu),$ and similar to the proof
of Proposition \ref{3.3.1}(ii), we have $\mu(x*(x*y))=\mu(y*(y*x))$, hence $\mu$ is a right state operator.
\end{proof}

By Proposition \ref{3.3.1}(iii), every right state BCK-algebra is a left state BCK-algebra.
In the following example, we show that the converse statement is not true, in general.  We present a left state operator $\mu$ on a BCK-algebra $X$ which is not a right state operator because $\Ker(\mu)$ is not a commutative ideal of $X.$

\begin{exm}\label{counter}
Let $X=\{0,1,2,3\}$. Define a binary operation $*$ by the following
table:
\[\begin{array}{ll}
\begin{tabular}{l|cccc} 
 $*$ & 0 & 1 & 2 & 3 \\\hline
 0 & 0 & 0 & 0 & 0\\
 1 & 1 & 0 & 0 & 0 \\
 2 & 2 & 2 & 0 & 0\\
 3 & 3 & 3 & 3 & 0 \\
 \end{tabular}
\end{array}\]
Then $(X,*,0)$ is a positive implicative BCK-algebra $(P)B_{4-1-4}$ from \cite{BCK} which is a chain ($0\leq 1\leq
2\leq 3$).
Let $\mu:X\ra X$ be defined by $\mu(0)=\mu(1)=0$ and $\mu(2)=\mu(3)=2$. We
claim that $\mu$ is
a left state operator on $X$. Clearly, it is a well defined and order
preserving map.
Let $x,y\in X$.

(1) If $x\leq y$, then we have $\mu(x*y)=\mu(0)=0$ and
$\mu(x)*\mu(x*(x*y))=\mu(x)*\mu(x)=0$.

(2) If $y<x$, then by definition of $*$, $\mu(x*y)=\mu(x)$. Also,
$\mu(x)*\mu(x*(x*y))=\mu(x)*\mu(x*x)=\mu(x)*\mu(0)=\mu(x)$.

(3) It can be easily shown that $\mu(x)*\mu(y)=\mu(\mu(x)*\mu(y))$.

From (1)--(3), we conclude that $\mu$ is a left state operator on $X$. But $\Ker(\mu)$ is not a commutative ideal of $X$ because $2*3 \in \Ker(\mu),$ but $2*(3*(3*2)) = 2*(3*3) = 2*0=2 \notin \Ker(\mu).$   Hence, $\mu$ is not a right state operator on $X.$
\end{exm}

Let $X$ be a set, we denote by $\Id_X: X \to X$ the identity on $X.$ It also provides an example of a left state operator which is not necessarily a right state operator.

In each BCK-algebra $X$, $\Id_X$ is a left state operator.
In fact, $\Id_X(x)*\Id_X(x*(x*y))=x*(x*(x*y))=x*y$. On the other hand, $\Id_X$ is a right state operator iff $X$ is a commutative BCK-algebra. So it can be easily obtained that,
$X$ is a commutative BCK-algebra if and only if each left state operator on $X$ is a right state operator.

By Proposition \ref{3.3.1}, each right state BCK-algebra is a left state BCK-algebra.
So in the remainder of this paper, we will consider only left BCK-algebras. Moreover, we write simply
a state BCK-algebra instead of a left state BCK-algebra.

\begin{defn}\label{3.4}
Let $(X,\mu)$ be a state BCK-algebra. An ideal $I$ of a BCK-algebra
$X$ is called a {\it state ideal} if
$\mu(I)\s I$. If $T$ is a subset of $X$, then $\langle T\rangle_{s}$ is the least
state ideal of $X$ containing $T$. A state ideal $I$ is said to be a {\it maximal state ideal} if $\langle I\cup\{x\}\rangle_s = X$ for each $x \in X-I.$ We denote by ${\rm MaxS}(X,\mu)$ the set of all maximal state ideals of $(X,\mu).$

\end{defn}

\begin{prop}\label{3.4.1}
Let $I$ be a state ideal of a state BCK-algebra $(X,\mu)$ and $a\in X$. Then
$$\langle I\cup\{a\}\rangle_{s}=\{x\in X \mid (x*a^n)*\mu(a)^m\in I \mbox{ for some $m,n\in \mathbb{N}$}\}.$$
\end{prop}

\begin{proof}
Set $A=\{x\in X \mid (x*a^n)*\mu(a)^m\in I \mbox{ for some $m,n\in \mathbb{N}$}\}$.
Clearly, $I\cup \{a\}\s A$. Moreover, if $J$ is a state ideal
of $(X,\mu)$ containing $I$ and $a$, then by Theorem \ref{2.2}, $A\s J$.
It suffices to show that $A$ is a state ideal. Let $x,y*x\in A$.
Then there are $m,n,s,t\in \mathbb{N}$ such that $(x*a^n)*\mu(a)^m\in I$ and $((y*x)*a^s)*\mu(a)^t\in I$.
\begin{eqnarray*}
(((y*a^{n+s})*\mu(a)^{m+t})&*&((x*a^n)*\mu(a)^m))*(((y*x)*a^s)*\mu(a)^t)\\
&\leq& (((y*a^{n+s})*\mu(a)^{t})*(x*a^n))*(((y*x)*a^s)*\mu(a)^t), \mbox{ by (BCK9)}\\
&=& (((y*a^{n+s})*(x*a^n))*\mu(a)^{t})*(((y*x)*a^s)*\mu(a)^t), \mbox{ by (BCK6)}\\
&\leq& (((y*a^{s})*x)*\mu(a)^{t})*(((y*x)*a^s)*\mu(a)^t), \mbox{ by (BCK9)}\\
&=& (((y*x)*a^{s})*\mu(a)^{t})*(((y*x)*a^s)*\mu(a)^t), \mbox{ by (BCK6)}\\
&=&0\in I.
\end{eqnarray*}

Since $(x*a^n)*\mu(a)^m,((y*x)*a^s)*\mu(a)^t\in I$ and $I$ is an ideal of $X$, then we get $(y*a^{n+s})*\mu(a)^{m+t}\in I$ and so
$y\in A$. Hence, $A$ is an ideal. Now, let $x$ be an arbitrary element of $A$. Then there exist $m,n\in\mathbb{N}$ such that
$(x*a^n)*\mu(a)^m\in I$. Since $I$ is a state ideal, then $\mu((x*a^n)*\mu(a)^m)\in I$ and so by Proposition \ref{3.2}(ii),
$\mu(x)*\mu(a)^{n+m}=(\mu(x)*\mu(a)^n)*\mu(a)^m=(\mu(x)*\mu(a)^n)*\mu(\mu(a))^m\in I$. Thus, $\mu(x)\in A$. Therefore,
$A$ is a state ideal of $(X,\mu)$.
\end{proof}

Note that, if $(X,\mu)$ is a state BCK-algebra, then $\{0\}$ and $X$ are state ideals of $(X,\mu)$ and so by
Proposition \ref{3.4.1}, $J=\{x\in X \mid (x*a^n)*\mu(a)^m=0 \ \mbox{\rm
for some } m,n\in \mathbb{N}\} $
is a state ideal of $X$ for any $a\in X$. Similarly,  we can construct other state ideals of $(X,\mu)$.

\begin{cor}\label{3.4.2}
A state ideal $I$ of a state BCK-algebra $(X,\mu)$ is a maximal state ideal if and only if
$\{x\in X \mid (x*a^n)*\mu(a)^m\in I \mbox{ for some } m,n\in \mathbb{N}\}=X$ for all $a\in X-I$.
\end{cor}

\begin{proof}
The proof is a straightforward corollary of  Proposition \ref{3.4.1}.
\end{proof}

By \cite[Thm 3.7]{BZ}, we know that if $M$ is a maximal ideal of $X$, then $I\cap J\s M$ implies that $I\s M$ or $J\s M$
for all $I,J\in {\rm I}(X)$. In the next theorem, we show that if $M$ is a maximal state ideal of  a state
BCK-algebra $(X,\mu)$, then $I\cap J\s M$ implies that $I\s M$ or $J\s M$ for all state ideals $I$ and $J$ of $(X,\mu)$.

\begin{thm}
Let $M$ be a maximal state ideal of a state BCK-algebra $(X,\mu)$. For for all state ideals $I$ and $J$ of $(X,\mu)$, we have $I\cap J\s M$
implies that $I\s M$ or $J\s M.$
\end{thm}

\begin{proof}
Let $I$ and $J$ be two state ideals of $(X,\mu)$ such that $I\cap J\s M$. Suppose that
there are $x\in I-M$ and $y\in J-M$. Then by Corollary \ref{3.4.2},
$X=\langle M\cup\{x\}\rangle_{s}=\langle M\cup\{y\}\rangle_{s}$. On the other hand,
if $a\in \langle M\cup\{x\}\rangle_{s}\cap \langle M\cup\{y\}\rangle_{s}$, then by Proposition \ref{3.4.1},
there exist $m,n,s,t\in\mathbb{N}$ such that $(a*x^n)*\mu(x)^m=m_1\in M$ and $(a*y^s)*\mu(y)^t=m_2\in M$
and so by (BCK2), (BCK4), (BCK6) and Proposition \ref{3.4.1}, we get $(a*m_1)*m_2\in I\cap J\s M$
(since $I$ and $J$ are  state ideals, $x\in I$, $y\in J$ and $(((a*m_1)*m_2)*x^n)*\mu(x)^m=0*m_2=0$,
$(((a*m_1)*m_2)*y^s)*\mu(y)^t=0*m_1=0$). Since $m_1,m_2\in M$ and $M$ is an ideal of $X$, then
we have $a\in M$ and so $X=\langle M\cup\{x\}\rangle_{s}\cap \langle M\cup\{y\}\rangle_{s}\s M$, which
is a contradiction. Therefore, $I\s M$ or $J\s M$.
\end{proof}

In Theorem \ref{con-ideal}, we  show a one-to-one relationship between congruence relations of a state BCK-algebra
$(X,\mu)$ and state ideals of $(X,\mu)$. We denote by ${\rm SI}(X)$ and $\Con(X,\mu)$ the set of state ideals and the set of congruences, respectively, on a state BCK-algebra $(X,\mu).$

\begin{thm}\label{con-ideal}
 Let $(X,\mu)$ be a state BCK-algebra.
 \begin{itemize}
  \item[{\rm (i)}] If $\theta$ is a congruence relation of $(X,\mu)$, then
  $[0]_{\theta}=\{x\in X \mid (x,0)\in \theta\}$ is a state ideal of $(X,\mu).$
  \item[{\rm (ii)}] If $I$ is a state ideal of $(X,\mu)$, then $\theta_{I}=\{(x,y)\in X\times X \mid  x*y,y*x\in I\}$
  is a congruence relation on $(X,\mu).$
  \item[{\rm (iii)}] There is a bijection between the set of all congruence relations of $(X,\mu)$, ${\rm Con}(X,\mu)$, and
  the set ${\rm SI}(X,\mu)$ of all state ideals of $(X,\mu)$.
 \end{itemize}
\end{thm}

\begin{proof}
(i) Let $\theta$ be a congruence relation of $(X,\mu)$.
Then by Theorem \ref{con BCK}, $[0]_{\theta}$ is an ideal of $X$. It suffices to show that $[0]_{\theta}$ is a state ideal. Let $x\in [0]_{\theta}$. Then $(x,0)\in \theta$. Since $\theta$ is a congruence relation of $(X,\mu)$, then $(\mu(x),\mu(0))\in\theta$ and so by Proposition \ref{3.2}(i), $(\mu(x),0)\in \theta$. Hence, $\mu(x)\in [0]_{\theta}$. That is, $[0]_{\theta}$ is a state ideal.

(ii) Let $I$ be a state ideal of $X$. Then $\theta_{I}$ is a congruence relation on a BCK-algebra $X$. Let
$(x,y)\in \theta_{I}$. Then $x*y,y*x\in I$ and so by Proposition \ref{3.2}(ii),
$\mu(x)*\mu(y)\leq\mu(x*y)\in I$. Thus, $\mu(x)*\mu(y)\in I$. In a similar way, $\mu(y)*\mu(x)\in I$, hence $(\mu(x),\mu(y))\in \theta_{I}$, so $\theta_{I}$ is a congruence relation of $(X,\mu)$.

(iii) We define a map $f:{\rm SI}(X,\mu)\ra {\rm Con}(X,\mu)$  by $f(I)=\theta_{I}$. Then it can be easily shown that $f$ is a bijection map and its inverse is the map  $g:{\rm Con}(X,\mu)\ra {\rm SI}(X,\mu)$, which is defined by $g(\theta)=[0]_{\theta}$.
\end{proof}

\begin{defn}\cite{Burris}
An algebra $A$ of type $F$ is a {\it subdirect product} of an indexed family $\{A_{i}\}_{i\in I}$ of algebras of type $F$ if
\begin{itemize}
\item $A$ is a subalgebra of $\Pi_{i\in I} A_{i}$,
\item $\pi_{i}(A)=A_{i}$ for any $i\in I$, where $\pi_i:\Pi_{i\in I} A_{i}\ra A_i$ is a natural projection map.
\end{itemize}
A  one-to-one homomorphism $\alpha : A\ra \Pi_{i\in I} A_{i}$ is called a {\it subdirect embedding}
if $\alpha (A)$ is a
subdirect product of the family $\{A_{i}\}_{i\in I}$. An algebra $A$ of type $F$ is
called {\it subdirectly irreducible} if, for every subdirect embedding $\alpha : A\ra \Pi_{i\in I} A_{i}$, there exists
$i\in I$ such that $\pi_{i}\circ\alpha: A\ra A_{i}$ is an isomorphism.
\end{defn}

\begin{rmk}\label{blyth1}
If $I$ and $J$ are two ideals of $X$ such that $I\s J$, then clearly, $\theta_I\s \theta_J$.
Let $(X,\mu)$ be subdirectly irreducible. Then by \cite[Thm II.8.4]{Burris}, the set ${\rm Con}(X,\mu)-\Delta$ has a least element, where $\Delta=\{(x,x) \mid x\in X\}$ and $\nabla=X\times X$.
Suppose that $\theta$ is the least element of ${\rm Con}(X,\mu)-\Delta$.
Then by Theorem \ref{con-ideal}, there exists
a state ideal of $(X,\mu)$ such that $\theta=\theta_{I}$ (so $I$ is a non-zero ideal of $X$). It follows that
$I$ is the least non-zero state ideal of $(X,\mu)$. By Theorem \ref{con-ideal} and
\cite[Thm II.8.4]{Burris},  we conclude that $(X,\mu)$ is subdirectly irreducible if and only if
${\rm SI}(X,\mu)-\{0\}$ has  the least element.
\end{rmk}

In Theorem \ref{irr} and Theorem \ref{characterize}, we present  characterizations of subdirectly irreducible state BCK-algebras.
First, we show that if $(X,\mu)$ is subdirectly irreducible, then the conditions (i) or (ii) of Theorem \ref{irr} hold. Then we prove that
if $(X,\mu)$ satisfies the condition (i) or (ii) in Theorem \ref{irr}, then $(X,\mu)$ must be subdirectly irreducible. We note that in the next theorem, we take an element $a$ in the subalgebra $\mu(X)$ of a BCK-algebra $X,$ therefore, $\langle a\rangle_X$ will denote the ideal of $X$ generated by the element $a.$

\begin{thm}\label{irr}
Let $(X,\mu)$ be a subdirectly irreducible state BCK-algebra.
\begin{itemize}
\item[{\rm (i)}] If $\Ker(\mu)=\{0\},$  then $\mu(X)$ is a subdirectly irreducible subalgebra of $X$.
\item[{\rm (ii)}] If $\Ker(\mu)\neq \{0\}$, then $\Ker(\mu)$ is a subdirectly irreducible subalgebra of $X$
and $\Ker(\mu)\cap \langle a\rangle_X \neq \{0\}$ for each non-zero element $a$ of $\mu(X)$.
\end{itemize}
\end{thm}

\begin{proof}
(i) Let $(X,\mu)$ be subdirectly irreducible and $\Ker(\mu)=\{0\}$. By Remark \ref{blyth1}, the set of
all non-zero state ideals of $(X,\mu)$ has the least element, $I$ say.
If $I\cap \mu(X)=\{0\}$, then by $\mu(I)\s I\cap \mu(X)$ (since $I$ is a state ideal), we conclude that
$\mu(x)=0$ for all $x\in I$. Thus, $I\s \Ker(\mu)=\{0\}$, which is a contradiction. So, $I\cap\mu(X)\neq\{0\}$.
Now, we show that $I\cap \mu(X)$ is the least non-zero ideal of $\mu(X)$.
Suppose that $J$ is an ideal of $\mu(X)$.

(1) Let $\langle J\rangle_X$ be the ideal of $X$ generated by $J,$ and choose an arbitrary element $x\in \langle J\rangle_X$.
Then by Theorem \ref{2.2}, there exist $b_1,\ldots ,b_n\in J$ such that
$(\cdots ((x*b_1)*b_2)* \cdots)*b_n=0$ and so by Proposition \ref{3.2}(i) and (ii), we get
$$(\cdots ((\mu(x)*\mu(b_1))*\mu(b_2))*\cdots)*\mu(b_n)\leq \mu((\cdots((x*b_1)*b_2)*\cdots)*b_n)=0.$$
Since $\mu^{2}=\mu$ and
$b_1,\ldots,b_n\in J\s \mu(X)$, we get $(\cdots ((\mu(x)*b_1)*b_2)*\cdots )*b_n=0$, hence $\mu(x)\in J$. Thus,
$\langle J\rangle_X$ is a state ideal of $(X,\mu)$.

(2) Clearly, $J=\langle J\rangle_{X} \cap \mu(X)$.

By (1), we get that $I\s \langle J\rangle_X$ and so by (2), $I\cap \mu(X)\s \langle J\rangle_{X} \cap \mu(X)=J$.
Hence, $I\cap\mu(X)$ is the least non-zero ideal of $\mu(X)$. Therefore, by \cite[Thm II.8.4]{Burris},
we conclude that $\mu(X)$ is a subdirectly irreducible subalgebra of $X$.

(ii) Let $\mu(X) \ne \{0\}.$ Again, let $I$ be the least non-zero state ideal of the subdirectly irreducible state BCK-algebra $(X,\mu).$ Since $X$ is a BCK-algebra, then every ideal of $X,$ in particular $\Ker(\mu),$ is a subalgebra of $X$.
Clearly, $\Ker(\mu)$ is a state ideal of $(X,\mu)$ and so $I\s \Ker(\mu)$. We show that
$I$ is the least non-zero ideal of $\Ker(\mu)$. Let $J$ be a non-zero ideal of $\Ker(\mu)$.
Then $\mu(J)\s \mu(\Ker(\mu))=\{0\}\s J$. For any $x,y\in X$, if $y*x,x\in J$,
then by Proposition \ref{3.2}(ii),
$0=\mu(y*x)\geq \mu(y)*\mu(x)=\mu(y)*0=\mu(y)$. Thus, $y\in \Ker(\mu)$, so $y\in J$ (since $J$ is an ideal  of $\Ker(\mu)$). It follows that $J$ is a state ideal of $X$ and so $I\s J$.
Hence by \cite[Thm II.8.4]{Burris}, $\Ker(\mu)$ is subdirectly irreducible.

Now, let $a$ be a non-zero element
of $\mu(X)$ and let $\langle a\rangle_X$ be the ideal generated by $a$ in $X$. Then $a=\mu(a).$  Take an arbitrary element $u\in \langle a\rangle_X$. By Theorem \ref{2.2}, there exists $n\in\mathbb{N}$ such that $0=u*a^n,$ and by by Proposition \ref{3.2}(ii),  $0=\mu(0)=\mu(u*a^n)=\mu(x)*(\mu(a))^n=\mu(u)*a^n.$
Thus, $\mu(u)\in \langle a\rangle_X$ and $\mu(\langle a\rangle_X)\s \langle a\rangle_X$. This implies, $\langle a\rangle _X$ is a non-zero state interval of $(X,\mu)$ and, consequently, $I\subseteq \langle a\rangle _X.$ Since also $I \subseteq \Ker(\mu),$ we have $\{0\}\ne I \subseteq \Ker(\mu) \cap \langle a\rangle_X.$
\end{proof}

\begin{thm}\label{characterize}
Let $(X,\mu)$ be a state BCK-algebra. If it satisfies the condition {\rm (i)} or {\rm (ii)} in Theorem {\rm \ref{irr}}, then $(X,\mu)$ is subdirectly irreducible.
\end{thm}

\begin{proof}
First, we assume that $\Ker(\mu)=\{0\}$ and $\mu(X)$ is a subdirectly irreducible subalgebra of $X$.
Since $\mu(X)$ is subdirectly irreducible, then by \cite[Thm II.8.4]{Burris},
$\bigcap({\rm I}(\mu(X))-\{0\})$ is a non-zero ideal of $\mu(X)$. From ${\rm I}(\mu(X))=\{I\cap \mu(X) \mid I\in {\rm I}(X)\}$,
it follows that  $\bigcap(\{I\cap \mu(X)\mid I\in {\rm I}(X)\}-\{0\})$ is non-zero, so $\bigcap({\rm I}(X)-\{0\})\neq\{0\}$.
Hence, the intersection of all non-zero state ideals of $(X,\mu)$ is a non-zero ideal of $X$ (clearly it is a
state ideal), whence by Remark \ref{blyth1}, $(X,\mu)$ is subdirectly irreducible.

Now, let $\Ker(\mu)\neq \{0\}$ and let $\Ker(\mu)$ be a subdirectly irreducible subalgebra of $X$.
Let $I$ be the least non-zero ideal of $\Ker(\mu).$
Clearly, $I$ is a state ideal (since $\mu(I)=\{0\}$). We claim, for any non-zero state ideal $H$ of $(X,\mu),$ we have $I\s H.$ Suppose that $H$ is a non-zero state ideal of $(X,\mu)$.
Then $\mu(H)\s H$. If $\mu(H)=\{0\}$, then $H\s \Ker(\mu)$ and so $I\s H$. Otherwise, there exists
$a\in \mu(H)-\{0\}$. It follows that $\{0\}\neq \Ker(\mu)\cap $ $\langle a \rangle_X \subseteq \Ker(\mu)\cap H$ and so
$I\s H\cap\Ker(\mu)\s H$. Thus, $I$ is the least non-zero
state ideal of $(X,\mu)$. Therefore, $(X,\mu)$ is subdirectly irreducible.
\end{proof}

In the final theorem of this section, we  find a relation between state operators in
BCK-algebras and MV-algebras. It is well known, if $(X,*,0)$ is a bounded commutative BCK-algebra,
then $(X,\oplus,',0)$ is an MV-algebra,
where $x\oplus y=N(Nx*y)$ and $x'=Nx$ for all $x,y\in X$ (see \cite{Mun}).
Note that in each bounded BCK-algebra
$X$, we have $N(Nx)=x$.

\begin{thm}\label{3.8}
Let $(X,*,0)$ be a bounded commutative BCK-algebra and $\mu$ be a left state BCK operator on $X$
such that $\mu(1)=1$. Then $(X,\mu)$ is a state MV-algebra. The converse is also true.
\end{thm}

\begin{proof}
Let $x,y\in X$. Then $\mu(x')=\mu(1*x)=\mu(1)*\mu(x*(x*1))=1*\mu(x)=\mu(x)'$. Then
\begin{eqnarray*}
\mu(x)\oplus \mu(y\ominus (x\odot y))&=&\mu(x)\oplus \mu((y'\oplus (x\odot y))')\\
&=& \mu(x)\oplus \mu(y*(x\odot y))\\
&=& \mu(x)\oplus \mu(y*(y\odot x))\\
&=& \mu(x)\oplus \mu(y*(y'\oplus x')')\\
&=& \mu(x)\oplus \mu(y*(y*x'))\\
&=& (\mu(x)'*\mu(y*(y*Nx)))'\\
&=& (\mu(Nx)*\mu(y*(y*Nx)))'\\
&=& \mu(Nx*y)', \mbox{ since $X$ is commutative and $\mu$ is a left state operator}\\
&=&\mu(N(Nx*y))\\
&=&\mu(x\oplus y),
\end{eqnarray*}
so that, $(X,\mu)$ is a state MV-algebra. Conversely, consider the MV-algebra $(X,\oplus,',0)$. If
$(X,\sigma)$ is a state MV-algebra, then we can easily show that $\sigma:X\ra X$ is a
left state operator on a BCK-algebra $(X,*,0),$ where $x*y := x\odot y',$ $x,y \in X.$ In fact, it follows from the following identity on $X$:
$$(y'\oplus (x'\odot y))'=y*(x'\odot y)=y*(y''\odot x')=y*(y'\oplus x)'=y*(y*x).$$
\end{proof}

\section{State-morphism BCK-algebras}

In the section, we introduce and study state-morphism BCK-algebras which is an important subfamily of the family of state BCK-algebras.

\begin{defn}\label{4.1}
Let $(X,*,0)$ be a BCK-algebra. A homomorphism $\mu:X\ra X$ is called a {\it state-morphism operator} if $\mu^{2}=\mu$,
where $\mu^2=\mu\circ \mu,$ and the pair $(X,\mu)$ is called a {\it state-morphism BCK-algebra}.
\end{defn}

By (BCK8), every state-morphism BCK-algebra is a (left) state BCK-algebra. We note that not every state-morphism operator is also a right state operator. For example, $\Id_X$ is both a state-morphism operator and a left state operator, but it is  a right state operator iff $X$ is a commutative BCK-algebra.

\begin{exm}\label{4.2}
(i) For each BCK-algebra $X$, the identity map $\Id_X:X\ra X$ and the zero operator $O_X(x)=0,$ $x \in X,$ are  state-morphism operators.

(ii) Let $x$ be an element of $X$ such that $a*x=a*x^2$ for all $a\in X$.
Define $\alpha_x:X\ra X$ by $\alpha_x(a)=a*x$ for all $a\in X$. First, we show that $\alpha_x$ is a homomorphism.
By (BCK4), $0*x=0$ for all $x\in X$.
Let $a,b\in X$. Then by (BCK6), we have  $(b*x)*b=(b*b)*x=0*x=0$, so
$b*x\leq b$. Using (BCK6) and (BCK7), we obtain that $(a*b)*x=(a*x)*b\leq (a*x)*(b*x)$.
On the other hand, by (BCK1) and (BCK6),  $(a*x)*(b*x)=(a*x^2)*(b*x)\leq (a*x)*b=(a*b)*x$. Hence, $\alpha_x$ is a homomorphism. Therefore,
$$
\alpha_x(\alpha_x(a))=(a*x)*x=a*x=\alpha_x(a),
$$
so that, $\alpha_x$ is a state-morphism operator on $X$. For example, if $x =0,$ then $\alpha_0={\rm Id}_X.$ In particular, if $X$ is a
positive implicative BCK-algebra, then by \cite[Thm 3.1.1]{chang}, for all $a,x\in X$, we have $a*x^2=a*x$ and so,
$\alpha_x$ is a state-morphism operator on $X$ for all $x\in X$.

(iii) {\it Every state operator $\mu$ on a linearly ordered commutative BCK-algebra $X$ is a state-morphism operator.} Indeed, if $x\le y,$ then $x*y=0$ and by Proposition \ref{3.2}, we have $0\le\mu(x)*\mu(y) \le \mu(x*y)=\mu(0).$ If $y \le x$, by the definition of a state operator, we have $\mu(x*y)=\mu(x)*\mu(x*(x*y))=\mu(x)*\mu(y)$. The both cases entail $\mu$ is an endomorphism of the BCK-algebra $X.$

(iv) {\it Every right state operator $\mu$ on a linearly ordered commutative BCK-algebra $X$ is a state-morphism operator.}
Indeed, by Proposition \ref{3.3.1}(iii), $\mu$ is a left state operator, too. Take $x,y\in X$. Since $X$ is a chain, then $x*y=0$ or $y*x=0$.

If $x*y=0$, then by Proposition \ref{3.2}(ii), $0\leq
\mu(x)*\mu(y)\leq\mu(x*y)=\mu(0)=0$. So,
$\mu(x)*\mu(y)=\mu(x*y)$.

If $y*x=0$, then from Proposition \ref{3.3.1} we have $\mu(x*y)=\mu(x)*\mu(y).$ Consequently, $\mu$ is a homomorphism, and $\mu$ is a state-morphism operator on $X.$
\end{exm}

Example \ref{counter} shows that if $\mu$ is a left state operator on a linearly ordered BCK-algebra, then $\mu$ is not necessarily a state-morphism operator on $X.$ Indeed, we have $\mu(3*2)=\mu(3)=2\neq 0=2*2=\mu(3)*\mu(2)$. Thus $\mu$ is not a state morphism operator.

As a consequence of Corollary \ref{3.8}, we have by \cite{DD1}, that  there are also bounded commutative BCK-algebras $X$ having state operators which are not  state-morphism operators.

\begin{defn}\label{4.3}
Let $(X,\mu)$ be a state-morphism BCK-algebra. An ideal $I$ of a BCK-algebra $X$ is called a {\it state ideal} if
$\mu(I)\s I$. If $T$ is a subset of $X$, then $\langle T\rangle_{s}$ is the least state ideal of $X$ containing $T$.
\end{defn}

It can be easily shown that, if $(X,\mu)$ is a state BCK-algebra, then $\Ker(\mu)$ is a state ideal of $X$.
Clearly, the intersection of every arbitrary family of state ideals of $X$ is a state ideal. So,
$$\langle T\rangle_{s}=\bigcap\{I \mid  T\s I, \  I \ \mbox{\rm is a state ideal of}\ (X,\mu)\}.$$

\begin{prop}\label{4.4}
Let $I$ be an ideal of a state-morphism BCK-algebra $(X,\mu)$. Then
$$\langle I\rangle_{s}=\{a\in X \mid (\cdots((a*\mu(x_{1}))*\mu(x_{2}))*\cdots)*\mu(x_n)\in I, \quad \exists\, n\in \mathbb{N},\quad \exists\, x_1,x_2,\ldots,x_n\in I\}.$$
\end{prop}

\begin{proof}
Let $J=\langle I\rangle_{s}=\{a\in X \mid (\cdots((a*\mu(x_{1}))*\mu(x_{2}))*\cdots)*\mu(x_n)\in I, \quad \exists\, n\in \mathbb{N},\quad \exists\, x_1,x_2,\ldots,x_n\in I\}$. Then clearly, $I\s J$ (since $0\in I$ and $\mu(0)=0$).
First, we show that $J$ is a state ideal of $X$. Let $a,b*a\in J$ for some
$a,b\in X$. Then there exist $m,n\in\mathbb{N}$ and $x_1,\ldots,x_n,y_1,\ldots,y_m\in X$ such that $(\cdots((a*\mu(x_{1}))*\mu(x_{2}))*\cdots)*\mu(x_n)\in I$
and $(\cdots(((b*a)*\mu(y_{1}))*\mu(y_{2}))*\cdots)*\mu(y_m)=y\in I$. By (BCK5) and (BCK6), we have
$$(\cdots(((b*y)*\mu(y_{1}))*\mu(y_{2}))*\cdots)*\mu(y_m)\leq a$$ and so by (BCK7),
$$(\cdots(((\cdots(((b*y)*\mu(y_{1}))*\mu(y_{2}))*\cdots)*\mu(y_m))*\mu(x_1))*\cdots)*\mu(x_n)\leq (\cdots((a*\mu(x_{1}))*\mu(x_{2}))*\cdots)*\mu(x_n)\in I.
$$
Since $y\in I$ and $I$ is an ideal of $X$, then by (BCK6),
$$(\cdots(((\cdots((b*\mu(y_{1}))*\mu(y_{2}))*\cdots)*\mu(y_m))*\mu(x_1))*\cdots)*\mu(x_n)\in I$$
and so $b\in J$. It follows that $J$ is an ideal of $X$. Moreover, if $c\in J$, then there
exist $n\in\mathbb{N}$ and $z_1,\ldots,z_n\in X$ such that $(\cdots((c*\mu(z_{1}))*\mu(z_{2}))*\cdots)*\mu(z_n)=z\in I$.
Hence, by (BCK5) and (BCK6),  we get that $((\cdots((\mu(c)*\mu(z_{1}))*\mu(z_{2}))*\cdots)*\mu(z_n))*\mu(z)=\mu(0)=0\in I$. Also,
$z_1,\ldots,z_n,z\in I$, so by definition of $J$, $\mu(c)\in J$. Thus, $\mu(J)\s J$ and so $J$ is a state ideal of $X$ containing $I$. Clearly,
if $K$ is a state ideal of $X$ containing $I$, then $J\s K$. Therefore, $J$ is the least state ideal of $X$ containing $I$. That is $J=\langle I\rangle_{s}$.
\end{proof}

\begin{prop}\label{4.5}
Let $(X,\mu)$ be a state-morphism BCK-algebra. Then the following hold:
\begin{itemize}
\item[{\rm (i)}] $\Ker(\mu)=\{x *\mu(x) \mid  x\in X\}=\{\mu(x)* x \mid  x\in X\}$.
\item[{\rm (ii)}] $X=\langle \Ker(\mu)\cup {\rm Im}(\mu)\rangle$.
\end{itemize}
\end{prop}

\begin{proof}
(i) Since $\mu^2=\mu$ and $\mu$ is a homomorphism, we have  $\{x*\mu(x) \mid  x\in X\}\s \Ker(\mu)$. Also, for each
$x\in \Ker(\mu)$, $x=x*0=x*\mu(x)\in \{x*\mu(x) \mid  x\in X\}$, so $\Ker(\mu)=\{x*\mu(x) \mid  x\in X\}$.
In a similar way, we can show that $\Ker(\mu)=\{\mu(x)*x \mid  x\in X\}$.

(ii) Let $x\in X$. By (i), $x*\mu(x)\in \Ker(\mu)$. Since $\mu(x)\in {\rm Im}(\mu)$, then by Theorem \ref{2.2},
we get that $x\in \langle \Ker(\mu)\cup {\rm Im}(\mu)\rangle$.
Therefore, $X=\langle \Ker(\mu)\cup {\rm Im}(\mu)\rangle$.
\end{proof}


Let $X$ be a bounded BCK-algebra and $m:X\ra [0,1]$ be a state-morphism. Since $m(1)=1$ and $m$ is an order preserving map, then $m(X)\s [0,1]$.
Therefore, $m$ is a homomorphism from the BCK-algebra $X$ to the BCK-algebra $([0,1],*_{\mathbb{R}},0)$.
Hence, $X/\Ker(m)$ and $m(X)$ are isomorphic.
By \cite[Thm 2.9]{Dvu1}, $\Ker(m)$ is a commutative ideal of $X$ and so $X/\Ker(m)$
is a bounded commutative BCK-algebra. Since
$([0,1],*_{\mathbb{R}},0)$ is a simple BCK-algebra and $m(X)$ is a subalgebra of it, then $m(X)$ is simple, so $X/\Ker(m)$ is simple, too.
It follows that $\Ker(m)$ is a maximal commutative ideal  of $X$. Therefore, $(X/\Ker(m),\oplus,',0/\Ker(m))$ is an MV-algebra, where
$x/\Ker(m)\oplus y/\Ker(m)=N(Nx*y)/\Ker(m)$ and $(x/\Ker(m))'=Nx/\Ker(m)$ for all $x,y\in X$.
It can be  easily shown that the map $f:X/\Ker(m)\ra [0,1]$ defined
by $f(x/\Ker(m))=m(x)$ is an MV-homomorphism and $X/\Ker(m)$ is a simple MV-algebra
(since $I$ is a BCK-ideal of $X/\Ker(m)$ if and only if $I$ is an
MV-ideal of $X/\Ker(m)$). By \cite[Thm 1.1]{Mun2},  there exists a unique one-to-one MV-homomorphism $\tau:X/\Ker(m)\ra [0,1]$. Thus,
$f=\tau$. By summing up the above results, we get that $m=\tau\circ\pi_{{\tiny \Ker}(m)}$,
where $\pi_{{\tiny \Ker}(m)}:X\ra X/\Ker(m)$ is the canonical epimorphism.
Conversely, let $X$ be a bounded BCK-algebra such that $X$ has at least one  commutative ideal, $I$ say. Then there exists a maximal ideal
$M$ of $X$ such that $I\s M$. In fact, $M$ is a maximal element of the set
$\{H\mid  \mbox{$H$ is an ideal of $X$ containing $I$, $1\notin H$}\}$.
Since $I$ is a commutative ideal and $I\s M$, then by \cite[Thm 2.5.2]{chang},
$M$ is a commutative ideal and so $X/M$ is a bounded commutative
simple BCK-algebra. It follows that $(X/M,\oplus,',0)$ is a
simple MV-algebra. By \cite[Thm 1.1]{Mun2},  there exists a unique
MV-homomorphism, $\tau_M:(X/M,\oplus,',0)\ra ([0,1],\oplus,',0)$.
Clearly, $\tau_M:X/M\ra [0,1]$ is a BCK-homomorphism and so
$\tau_M\circ\pi_M:X\ra [0,1]$ is a state-morphism, where $\pi_{M}:X\ra X/M$ is the canonical epimorphism.

Now, let $X$ be a bounded BCK-algebra and $\mu:X\ra X$ be a state-morphism operator on $X$ such that
$\Ker(\mu)$ is a commutative ideal of $X$. Then $X/\Ker(\mu)$ is a bounded commutative BCK-algebra.
Thus, $\mu(X)$ is an MV-algebra (since $\mu(X)\cong X/\Ker(\mu))$.
Suppose that $H$ is a maximal ideal of the MV-algebra $\mu(X)$ and $\pi_H:\mu(X)\ra \mu(X)/H$ is the canonical epimorphism.
Then $\mu(X)/H$ is a simple MV-algebra and so by \cite[Thm 1.1]{Mun2},
there is a unique MV-homomorphism $\tau_H:\mu(X)/H\ra [0,1]$. Clearly, $\tau_H\circ \pi_H\circ\mu: X\ra [0,1]$ is a
measure-morphism. Moreover, if $\mu(1)=1$, then  $\tau_H\circ \pi_H\circ\mu$ is a state-morphism.

\begin{rmk}\label{4.5.0}
Let $\mu$ be a state-morphism operator on $X$ such that $\Ker(\mu)=\{0\}$. Then for all $x\in X$, $x*\mu(x),\mu(x)*x\in\Ker(\mu)=\{0\}$ and so
by (BCK3), $\mu(x)=x$. Therefore, $\mu=\Id_X$.
\end{rmk}

\begin{cor}\label{4.5.1}
If $X$ is a simple BCK-algebra, then $\Id_X$ and $O_X$ are all state-morphism operators of $X.$
\end{cor}

\begin{proof}
Let $X$ be a simple BCK-algebra and $\mu:X\ra X$ be a state-morphism operator on $X$. Then $\Ker(\mu)=\{0\}$ or $\Ker(\mu)=X$.
Hence by Remark \ref{4.5.0}, $\mu=\Id_X$ or $\mu(x)=0$ for all $x\in X$.
\end{proof}

\begin{defn}
A state ideal $I$ of a state-morphism BCK-algebra $(X,\mu)$ is called a {\it prime state ideal} of $(X,\mu)$ if,
given state ideals $A,B$ of $(X,\mu)$, $A\cap B\s I$ implies that $A\s I$ or $B\s I$.
\end{defn}

\begin{thm}\label{4.6}
Let $(X,\mu)$ be a subdirectly irreducible state-morphism BCK-algebra. Then $\Ker(\mu)$ is a prime state ideal.
\end{thm}

\begin{proof}
Let $I$ and $J$ be two state ideals of $(X,\mu)$  such that $I\cap J\s \Ker(\mu)$.
Define $\phi:X/\Ker(\mu)\ra \mu(X)/I\times \mu(X)/J$, by $\phi(x/\Ker(\mu))=(x/I,x/J)$ for all $x\in X$.
For each $x,y\in X$, if $x/\Ker(\mu)=y/\Ker(\mu)$, then $x*y,y*x\in\Ker(\mu)$
and so $\mu(x)*\mu(y)=\mu(x*y)=0=\mu(y*x)=\mu(y)*\mu(x)$. Hence by
(BCK3), $\mu(x)=\mu(y)$. Therefore, $\phi$ is a well defined
homomorphism. Thus, for each $x,y\in X$, if
$\phi(x/\Ker(\mu))=\phi(y/\Ker(\mu))$, then $(\mu(x)/I,\mu(x)/J)=(\mu(y)/I,\mu(y)/J)$,
so that $\mu(x)*\mu(y),\mu(y)*\mu(x)\in I\cap J$. Hence,
$\mu(x)*\mu(y),\mu(y)*\mu(x)\in \Ker(\mu)$. It follows that $x/\Ker(\mu)=y/\Ker(\mu)$,
which implies that $\phi$ is one-to-one. Clearly,
$\pi_1\circ \phi(X/\Ker(\mu))=\mu(X)/I$, and $\pi_2\circ \phi(X/\Ker(\mu))=\mu(X)/J$,
where $\pi_1:\mu(X)/I\times \mu(X)/J\ra \mu(X)/I$ and
$\pi_2:\mu(X)/I\times \mu(X)/J\ra \mu(X)/J$ are natural projection maps.
Since $X/\Ker(\mu)$ and $\mu(X)$ are isomorphic, then by
Theorem \ref{irr}(ii), $X/\Ker(\mu)$ is a subdirectly irreducible BCK-algebra
and so $\pi_1\circ \phi:X/\Ker(\mu)\ra \mu(X)/I$ or
$\pi_2\circ \phi:X/\Ker(\mu)\ra \mu(X)/J$ is an isomorphism. Without lost of generality,
we can assume that  $\pi_1\circ \phi$ is an isomorphism.
For any $x\in I$, $\pi_1(\phi(x/\Ker(\mu)))=\pi_1(\mu(x)/I,\mu(x)/J)=\mu(x)/I$.
Since $I$ is a  state ideal, then $\mu(x)\in I$ and hence
$\mu(x)/I=0/I$. It follows that  $x/\Ker(\mu)=0/\Ker(\mu)$ (since $\pi_1\circ\phi$ is an isomorphism)
and $x\in \Ker(\mu)$. Therefore, $I\s \Ker(\mu)$ and so
$\Ker(\mu)$ is a prime ideal of $X$.
\end{proof}

Now, let us to consider a commutative subdirectly irreducible state
morphism BCK-algebra $(X,\mu)$ satisfying the identity
$(x*y)\wedge (y*x)=0$.
\begin{prop}\label{condition}
Let $(X,\mu)$ be a subdirectly irreducible state-morphism BCK-algebra such that $X$ is commutative and
$(x*y)\wedge (y*x)=0$ for all $x,y\in X$. Then the following statements conditions hold:
\begin{itemize}
 \item[{\rm (i)}] For all $x\in X$, either $x\leq \mu(x)$ or $\mu(x)\leq x.$
 \item[{\rm (ii)}] $\mu(X)$ is a chain.
 \end{itemize}
\end{prop}

\begin{proof}
(i) Since $(X,\mu)$ is subdirectly irreducible, then by Theorem \ref{irr}, $\Ker(\mu)=\{0\}$ or
$\Ker(\mu)\neq \{0\}$ and it is a subdirectly irreducible subalgebra of $X$. If $\Ker(\mu)=\{0\}$, then
by Remark \ref{4.5.0}, $\mu(x)=x$ for all $x\in X$. Let $\Ker(\mu)\neq\{0\}$.
Since $(x*y)\wedge (y*x)=0$ for all $x,y\in X$, then
by Theorem \ref{irr} and \cite[Thm 2.3.12]{chang}, $\Ker(X)$ must be a chain. Let $x\in X$. By Proposition \ref{4.5},
$x*\mu(x),\mu(x)*x\in \Ker(\mu)$ and so $(x*\mu(x))\wedge (\mu(x)*x)=0$ implies that $x*\mu(x)=0$ or $\mu(x)*x=0$. Therefore,
$x\leq \mu(x)$ or $\mu(x)\leq x$.

(ii) By the first isomorphism theorem, $X/\Ker(\mu)\cong \mu(X)$. Since $X$ is a commutative BCK-algebra and it satisfies the
identity $(x*y)\wedge (y*x)=0$, then by \cite[Thm II.8.13]{BCK} and Theorem \ref{4.6}, $X/\Ker(\mu)$ is
a chain. Hence, $\mu(X)$ is a chain.
\end{proof}

Note that if $(X,*,0)$ is a BCK-algebra such that $(X,\leq)$ is a lattice, it is called a {\it BCK-lattice}.
Then by \cite[Thm 2.2.6]{chang}, $X$ satisfies the identity $(x*y)\wedge (y*x)=0$.

\begin{defn}\label{adjoint}
A pair $(A,I)$ is called an {\it adjoint pair} of a BCK-algebra $X$, if $I$ is an ideal of $X$ and
$A$ is a subalgebra of $X$ satisfying
the following conditions:
\begin{itemize}
\item[(Ap1)]  $A\cap I=\{0\}$ and $\langle A\cup I\rangle=X;$
\item[(Ap2)] for each $x\in X$, there exists an element $a_x\in A$
such that $(x,a_x)\in \theta_{I}$ (we say that $a_x$ is a {\it component} of
$x$  in $A$ with respect to $I$).
\end{itemize}
\end{defn}

By Proposition \ref{4.5}(iii) and (iv), we conclude that if $\mu$ is a state-morphism operator on $X$,
then $(\mu(X),\Ker(\mu))$ satisfies (Ap1). In Theorem \ref{adjoint1}, a relation between
state-morphism operators and adjoint pairs in any BCK-algebras will be found.

\begin{prop}\label{com}
Let $(A,I)$ be an adjoint pair of $X$. Then, for all $x\in X$, $a_x$ is unique.
\end{prop}
\begin{proof}
Let $x\in X$ and $a,b\in A$ be two components of $x$ in $A$. Then $(x,a),(x,b)\in\theta_{I}$ and so
$(a,b)\in\theta_{I}$. Hence, $a*b,b*a\in I$. Also, $a*b,b*a\in A$ (since $A$ is a subalgebra of $X$), so by (Ap1), $a*b,b*a\in I\cap A=\{0\}$.
Thus, by (BCK3), $a=b$. Therefore, $a_x$ is the only component of $x$ in $A$ with respect to $I$.
\end{proof}

Let $\mu$ and $\nu$ be two state-morphism operators on $X$ such that $\Ker(\mu)=\Ker(\nu)$ and ${\rm Im}(\mu)={\rm Im}(\nu)$.
 For any $x\in X$, we have
 $x*\mu(x),\mu(x)*x\in \Ker(\mu)=\Ker(\nu)$ and so $\nu(x*\mu(x))=0=\nu(\mu(x)*x)$. Since $\nu$ is a homomorphism and
 $\mu(x)\in {\rm Im}(\mu)={\rm Im}(\nu)$, then $\nu(\mu(x))=\mu(x)$ and so $\nu(x)*\mu(x)=0=\mu(x)*\nu(x)$. From (BCK3), we obtain that
 $\nu(x)=\mu(x)$ for all $x\in X$. Therefore, $\mu=\nu$. In Remark \ref{4.7.0}, we show that, there are state-morphism operators
 $\mu$ and $\nu$ on a BCK-algebra $X$ such that $\Ker(\mu)=\Ker(\nu)$, but $\mu\neq\nu$.

\begin{rmk}\label{4.7.0}
 Suppose that $I$ is a maximal ideal of $X$ such that $|X/I|=2$ and $2\leq |X-I|$. Let $a$ and $b$ be two distinct elements
 of $X-I$. Define $\mu_a:X\ra X$ and $\mu_b:X\ra X$ by
\begin{equation*}
\mu_a(x)=\left\{\begin{array}{ll}
0 & \text{if $x\in I$},\\
a & \text{if $x\in X-I$}.\\
\end{array} \right.
\hspace{1cm} \mu_b(x)=\left\{\begin{array}{ll}
0 & \text{if $x\in I$},\\
b & \text{if $x\in X-I$}.\\
\end{array} \right.
\end{equation*}

(1) If $x,y\in I$, then $x*y\in I$, so $\mu_a(x*y)=0=\mu_a(x)*\mu_b(y).$

(2) If $x\in I$ and $y\in X-I$, then $x*y\leq x$ and hence $x*y\in I$. It follows that $\mu_a(x*y)=0=0*\mu_a(y)=\mu_a(x)*\mu_b(y)$,

(3) If $x\in X-I$ and $y\in I$, then $x*y\in X-I$ (since $I$ is an ideal and $x*y\in I$ implies $x\in I$) and so
$\mu_a(x*y)=a=\mu_a(x)*0=\mu_a(x)*\mu_a(y)$,

(4) If $x,y\in X-I$, then by assumption, $x/I=y/I$ (since $|x/I|=2$), so $x*y\in I$. Thus,
$\mu_a(x*y)=0=a*a=\mu_a(x)*\mu_a(y)$.

By (1)-(4), we obtain that $\mu_a$ is a homomorphism. If $x\in I$, then $\mu_a(\mu_a(x))=\mu_a(x)=0$.
Also, if $x\in X-I$, then $\mu_a(\mu_a(x))=\mu_a(a)=a=\mu_a(x)$ (since $a\in X-I$), so $\mu_a$ is a state-morphism
operator. In a similar way, we can show that $\mu_b$ is a state-morphism operator. Clearly,
$\Ker(\mu_a)=I=\Ker(\mu_b)$, but $\mu_a\neq\mu_b$.
\end{rmk}

Note that if $X$ is a non-trivial positive implicative BCK-algebra and
$I$ is a maximal ideal of $X$, then $X/I$ is a simple positive implicative BCK-algebra and so
by \cite[Cor 3.1.7]{chang}, $|X/I|=2$. It follows that if $2\leq |X-I|$, then $X$ satisfies the conditions
in Remark \ref{4.7.0}.

\begin{thm}\label{adjoint1}
There is a one-to-one correspondence between adjoint pairs of $X$ and state-morphism operators on $X$.
\end{thm}

\begin{proof}
Let $\mu:X\ra X$ be a state-morphism operator on $X$. We show that $(\mu(X),\Ker(\mu))$ is an adjoint pair of $X$.
By Proposition \ref{4.5}(iii) and (iv), (Ap1) holds. Let $A=\mu(X)$ and $x$ be an element of $X$.
Then $\mu(x)\in A$ and clearly, $x*\mu(x),\mu(x)*x\in\Ker(\mu)$ (since $\mu^2=\mu$).
Hence, $(x,\mu(x))\in \theta_{I}$. That is, for each $x\in X$, $\mu(x)$ is a component of $x$ in $A$ and so
(Ap2) holds. Therefore, $(\mu(X),\Ker(\mu))$ is an adjoint pair of $X$.

Conversely, let $(A,I)$ be an adjoint pair of $X$. Define $\mu_{I,A}:X\ra X$,
by $\mu_{I,A}(x)=a_{x}$ for all $x\in X$. By Proposition \ref{com},
$\mu_{I,A}$ is well defined. Let $x,y\in X$. Then $(x,a_x)\in\theta_{I}$ and $(y,a_y)\in \theta_I$ and so $(x*y,a_x*a_y)\in\theta_I$. By
$a_x*a_y\in A$, we conclude that $a_x*a_y$ is a component of $x*y$ in  $A$, hence by Proposition \ref{com},  $\mu_{I,A}(x*y)=a_{x*y}=a_x*a_y=\mu_{I,A}(x)*\mu_{I,A}(y)$.
Thus, $\mu_{I,A}$ is a homomorphism. Moreover, for any $a\in A$, $a*a=0\in I$ and hence $\mu_{I,A}(a)=a_a=a$.
It follows that $\mu_{I,A}(\mu_{I,A}(x))=\mu_{I,A}(x)$ for all $x\in X$.
Therefore, $\mu_{I,A}$ is a state-morphism operator on $X$. Let us denote by
${\rm Ad}(X)$ and ${\rm SM}(X)$ the set of all adjoint pairs and the set of all state-morphism operators on $X$, respectively.
Define $f:{\rm Ad}(X)\ra {\rm SM}(X)$, by $f(A,I)=\mu_{I,A}$ and $g:{\rm SM}(X)\ra {\rm Ad}(X)$ by $g(\mu)=(\mu(X),\Ker(\mu))$.
Since $\Ker(\mu_{I,A})=I$ and ${\rm Im}(\mu_{I,A})=A$ for all $(A,I)\in Ad(X)$, then by the
paragraph just after Proposition \ref{com}, we conclude that $f\circ g=\Id_{\rm SM(X)}$ and $g\circ f=\Id_{\rm Ad(X)}$.
\end{proof}

In the sequel, we want to construct a state BCK-algebra from a state-morphism. Let $m:X\ra [0,1]$ be a state-morphism. Then
$m$ is a homomorphism from $X$ into the BCK-algebra $([0,1],*_{\mathbb{R}},0)$ and so $X/\Ker(m)\cong m(X)$. Let $B=m(X)$ and
$C=\Ker(m)$. Then $B$ and $C$ are BCK-algebras. Consider the BCK-algebra $B\times C$. Let $A=\{(b,0)|b\in B\}$ and
$I=\{(0,c)|c\in C\}$. Then $I$ is an ideal of $B\times C$ and $A$ is a subalgebra of $B\times C$. Also,

(1) $A\cap I=\emptyset$.

(2) For each $(x,y)\in B\times C$, we have $((x,y)*(x,0))*(0,y)=(0,0)$, hence by Theorem \ref{2.2},
$(x,y)\in \langle A\cup I\rangle$. It follows that $B\times C=\langle A\cup I\rangle$.

(3) For each $(x,y)\in B\times C$, we have $(x,y)*(x,0)=(0,y)\in I$ and $(x,0)*(x,y)=(0,0)\in I$. Thus,
$(x,y)/I=(x,0)/I$.

So by Theorem \ref{adjoint1}, the map $\mu:B\times C\ra B\times C$ defined by
 $\mu(x,y)=(x,0)$ is a state-morphism operator on $B\times C$.
 Clearly, $\Ker(\mu)=I$ and ${\rm Im}(\mu)=A$. Note that if $m_\mu:B\times C\ra [0,1]$
 is the state-morphism induced by $\mu$ (see the paragraph before Remark \ref{4.5.0}),
 then $(B\times C)/\Ker(m_\mu)\cong B\cong {\rm Im}(m)$ and $\Ker(m_\mu)=C\cong \Ker(m)$.


\begin{defn}
Let $I$ be an ideal of $X$ and $\pi_I:X\ra X/I$ be the canonical projection.
Then $I$ is called a {\it retract} ideal if there exists
a homomorphism $f:X/I\ra X$ such that $\pi_I\circ f=\Id_{X/I}$ (the identity map on $X/I$).
\end{defn}

\begin{thm}\label{retract}
 An ideal $I$ of $X$ is a retract ideal if and only if there exists a subalgebra $A$
 of $X$ such that $(A,I)$ forms an adjoint pair.
\end{thm}

\begin{proof}
Let $I$ be a retract ideal of $X$. Then there exists a homomorphism $f:X/I\ra X$
such that $\pi_I\circ f=\Id_{X/I}$. Put $A=f(X/I)$.
Since $f$ is a homomorphism, then $A$ is a subalgebra of $X$. Let $x\in I\cap A$.
Then there exists $a\in X$ such that $f(a/I)=x$, so
$a/I=\pi_I\circ f(a/I)=\pi_I(x)=x/I$. From $x\in I$, we get that $a\in I$ and $a/I=0/I$, whence $x=f(0/I)=0$. Now, let $x\in X$.
Then $f(x/I)=a$ for some $a\in A$. It follows that $x/I=\pi_I\circ f(x/I)=\pi_I(a)=a/I$,
which implies that $x*a\in I$. Hence, $a\in \langle A\cup I\rangle$
and $a$ is a component of $x$ in $A$ with respect to $I$. Therefore, $(A,I)$ is an adjoint pair of $X$. Conversely, let $(A,I)$ be an adjoint pair of $X$. Define $f:X/I\ra X$ by $f(x/I)=a_x$ for all $x\in X$ (see Definition \ref{adjoint}).
If $x/I=y/I$ for some $x,y\in X$,
then $(x,y)\in\theta_I$ and $(x,a_x)\in \theta_I,$ which yields $a_x$ is a component of $y$ in $A$. By Proposition \ref{com}, we get that
$a_y=a_x$. Thus, $f$ is well defined. In a similar way,
we can show that $f$ is a homomorphism.
It follows from $(x,a_x)\in\theta_I$ that $\pi_I\circ f(x/I)=\pi_I(a_x)=a_{x}/I=x/I$.
Therefore, $I$ is a retract ideal of $X$.
\end{proof}

\begin{cor}\label{ret-state}
There is a one-to-one correspondence between retract ideals and state-morphism operators of $X$.
\end{cor}

\begin{proof}
The proof is a straightforward  consequence of Theorem \ref{adjoint1} and \ref{retract}.
\end{proof}

\begin{defn}\label{4.7}\cite[Def II.8.8]{Burris}
A state BCK-algebra $(X,\mu)$ is called
\begin{itemize}
\item {\it simple} if ${\rm Con}(X,\mu)=\{\Delta,\nabla\}$.
\item {\it semisimple} if the intersection of all maximal congruence relations of $(X,\mu)$ is $\Delta$.
\end{itemize}
\end{defn}

By Theorem \ref{con-ideal}, we conclude that $(X,\mu)$ is simple if and only if it has exactly, two state ideals ($\{0\}$ and $X$) and it is
semisimple if and only if the intersection of all maximal state ideals of $(X,\mu)$ is the zero ideal.
\begin{thm}\label{4.10}
 Let $(X,\mu)$ be a state-morphism BCK-algebra. Then the following hold:
 \begin{itemize}
  \item[{\rm (i)}] $\mu(X)$ is a simple (semisimple) subalgebra of $X$ if and only if $\Ker(\mu)\in \Max(X)$ $(\Rad(X)\s\Ker(\mu))$.
  \item[{\rm (ii)}] $(X,\mu)$ is a simple state-morphism BCK-algebra if and only if $X$ is a simple BCK-algebra.
  \item[{\rm (iii)}] If $\mu(X)$ is a semisimple subalgebra of $X$, then the intersection of all maximal state ideals of $(X,\mu)$ is a subset of $\Ker(\mu)$.
  \item[{\rm (iv)}] If $X$ is a non-trivial bounded BCK-algebra such that $\mu(1)=1$ and $(X,\mu)$ is a semisimple state BCK-algebra, then  $\mu$ is the identity map.
 \end{itemize}
\end{thm}

\begin{proof}
(i) Let $(X,\mu)$ be a state-morphism BCK-algebra. Then by the first isomorphism theorem,
$X/\Ker(\mu)$ and $\mu(X)$ are isomorphic (as BCK-algebras), whence
the proof of (i) is straightforward.

(ii) Let $(X,\mu)$ be a simple state-morphism BCK-algebra. By Proposition \ref{3.2}(iii), $\Ker(\mu)$ is a state ideal of
$(X,\mu)$ and so $\Ker(\mu)=\{0\}$ or
$\Ker(\mu)=X$. By Corollary \ref{4.5.1}, we obtain that $\mu=\Id_X$ or $\mu(x)=0$ for all $x\in X$. However, each ideal of $X$ is a state ideal, so by assumption, $X$ must have exactly two ideals. That is, $X$ is a simple BCK-algebra. The proof of the converse is clear. In fact,
any simple BCK-algebra $X$, has exactly two ideals, $X$ and $\{0\}$, which are state ideals.

(iii) Let $\mu(X)$ be a semisimple subalgebra of $X$. Since $X/\Ker(\mu)\cong \mu(X)$, we get that
$\Rad(X/\mu(X))=\{0/\mu(X)\}$ and so $\bigcap\{I/\Ker(\mu) \mid  \Ker(\mu)\s I\in \MaxS(X)\}=\{0/\Ker(\mu)\}$, which implies that
$\bigcap\{I \mid  \Ker(\mu)\s I\in \MaxS(X)\}\s \Ker(\mu)$. Let
$H$ be a maximal ideal of $X$ containing $\Ker(\mu)$. Since $\mu(x)*x\in \Ker(\mu)$, for
each $x\in H$, then we have $\mu(x)*x\in H,$ and so $\mu(x)\in H$ for all $x\in X$.
Thus, $H$ is a state ideal of $(X,\mu)$.  By summing up the above results, we have
$$
\bigcap\{I \mid  \mbox{ $I$ is a state ideal of $(X,\mu)$}\}\s \bigcap\{I \mid  \Ker(\mu)\s I\in \MaxS(X)\}\s \Ker(\mu).
$$

(iv) Let $I$ be a maximal state ideal of $X$. Then we define $\nu:X/I\ra X/I$ by $\nu(x/I)=\mu(x)/I$
for all $x\in X$. If $x/I=y/I$ for some $x,y\in X$, then $x*y,y*x\in I$. By assumption,
$\mu(x)*\mu(y)\in I$ and $\mu(y)*\mu(x)\in I$, hence $\mu(x)/I=\mu(y)/I$, which implies that
$\nu(x/I)=\nu(y/I)$. Clearly, $\nu$ is a state operator on the BCK-algebra $X/I$. Since $I$ is a maximal
ideal, then $X/I$ is a simple BCK-algebra, so by Corollary \ref{4.5.1}, $\nu=\Id_{X/I}$ or $\nu=0$.
If $\nu=0$, then $\mu(x)\in I$ for all $x\in X$. It follows that $1\in I$, which is a contradiction. So, $\nu(x/I)=x/I$
for all $x\in X$. Hence, $\mu(x)*x,x*\mu(x)\in I$ for all $x\in X$. Since $I$ is an arbitrary maximal state ideal of $(X,\mu)$, then by  Proposition \ref{4.2},
we conclude that $\Ker(\mu)\s\bigcap\{I \mid I \in \MaxS(X)\}$. Now, let $(X,\mu)$ be semisimple. Then
$\bigcap\{I \mid I\in \MaxS(X)\}=\{0\}$ and so, $\Ker(\mu)=\{0\}$. By Corollary \ref{4.5.0},
$\mu=\Id_X$.
\end{proof}

Now we show a relation between state-morphism MV-algebras and state-morphism BCK-algebras.

\begin{thm}\label{4.12}
Let $(X,*,0)$ be a bounded commutative BCK-algebra and $\mu:X\ra X$ be a state-morphism operator such that
$\mu(1)=1.$ Then $(X,\mu)$ is a state-morphism MV-algebra.
\end{thm}

\begin{proof}
Let $x,y\in X$. Then $\mu(x')=\mu(1*x)=\mu(1)*\mu(x)=1*\mu(x)=\mu(x)'$. Also,
$$\mu(x\oplus y)=\mu(N(Nx*y))=1*\mu(Nx*y)=1*(\mu(Nx)*\mu(y))=1*((1*\mu(x))*\mu(y))=\mu(x)\oplus \mu(y)$$
so, $\mu(x)$ is a homomorphism of MV-algebras. Since $\mu^2=\mu$, then $\mu$ is a state-morphism operator
on the MV-algebra $(X,\oplus,',0)$. That is, $(X,\mu)$ is a state-morphism MV-algebra.
\end{proof}

\section{Generators of State-Morphism BCK-algebras}

Let $\mathcal {SMBCK}$ be the quasivariety of state-morphism BCK-algebras. We note that the system of BCK-algebras is not a variety because it is not closed under homomorphic images, \cite[Thm VI.4.1]{BCK}. On the other side, the system of commutative BCK-algebras or of quasi-commutative BCK-algebra forms a variety, \cite[Thm I.5.2, Thm I.9.2]{BCK}. Since by \cite[Thm I.9.4]{BCK}, every finite BCK-algebra is quasi-commutative, we can define the variety generated by a system of finite BCK-algebras.

Let $(X,*,0)$ be a BCK-algebra and on the direct product BCK-algebra $X\times X$ we define a mapping $\mu_X: X \times X \to X\times X$ by $\mu_X(x,y)=(x,x),$ $(x,y) \in X \times X.$  Then $\mu_X$ is a state-morphism on the BCK-algebra $X\times X$ and the state-morphism BCK-algebra $D(X):=(X\times X,\mu_X)$ is a said to be a {\it diagonal state-morphism BCK-algebra.} In the same way we can define also $\nu:X\times X \to X\times X$ by $\nu(x,y) =(y,y),$ $(x,y) \in X \times X,$ and $(X\times X,\nu)$ is again a state-morphism BCK-algebra which is isomorphic to $D(X)$ under the isomorphism $h(x,y)=(y,x),$ $(x,y) \in X \times X.$ For example, if $X =[0,1]$ is the MV-algebra of the real interval, then it generates the variety of MV-algebras (as well as the quasivariety of MV-algebras), and  by \cite[Thm 5.4]{DKM}, $D([0,1])$ generates the variety of state-morphism MV-algebras.

Given a quasivariety of BCK-algebras $\mathcal V,$ let $\mathcal V_\mu$
denote the class of state-morphism BCK-algebras $(X,\mu)$ such
that $X \in \mathcal V.$ Then $\mathcal V_\mu$ is a quasivariety, too.

As usual, given a class $\mathcal{K}$ of algebras of the same type,
$\mathsf{ I}(\mathcal{K})$, $\mathsf{H}(\mathcal{K})$,
$\mathsf{S}(\mathcal{K}),$ $\mathsf{P}(\mathcal{K}),$ and $\mathsf{P_R}(\mathcal{K})$ will
denote the class of isomorphic images, of homomorphic images, of
subalgebras, of direct products of algebras and of reduced products from $
\mathcal{K}$, respectively. Moreover, let $\mathsf{Q_V}(\mathcal{K})$ and $\mathsf{V}(\mathcal{K})$ denote the quasivariety and the variety, respectively, generated by $\mathcal{K}$. We recall that a quasivariety is closed under isomorphic images, subalgebras, reduced products and containing the one-element algebras, see \cite[Def V.2.24]{Burris}, and a variety is closed under homomorphic images, subalgebras and products.

Using methods from \cite[Sec 5]{DKM}, which can be easily modified for
state-morphism BCK-algebras instead of state-morphism MV-algebras, we
can  prove the following two lemmas and  theorem on generators for a case when we have a variety of BCK-algebras as well as for a more general case~- for quasivarieties of BCK-algebras; for
reader's convenience, we present outlines of theirs proofs.

First we start with proofs concerning the case when a family of BCK-algebras belongs to some variety of BCK-algebras.

\begin{lem}\label{wd}
{\rm (1)} Let $\mathcal{K}$ be a class of
BCK-algebras belonging to some variety of BCK-algebras. Then $\mathsf{V(D}(\mathcal{K}))\subseteq
\mathsf{V}(\mathcal{K})_\mu.$
\newline {\rm (2)} Let $\mathcal{V}$
be any variety of BCK-algebras. Then $\mathcal{V}
_\mu=\mathsf{ISD}(\mathcal{V})$.
\end{lem}

\begin{proof}
(1) We have to prove that every BCK-reduct of a state-morphism BCK-algebra in $\mathsf{V(D}(\mathcal{K}))$ is in
$\mathsf{V}(\mathcal{K})$. Let $\mathcal{K}_{0}$ be the class of all
BCK-reducts of algebras in $\mathsf{D}(\mathcal{K})$. Let $X \in
\mathcal K,$ then $D(X)\in \mathsf{D}(\mathcal{K}).$ Then the BCK-reduct of $D(X)$ is $X\times X$, and since $X$ is a homomorphic
image (under the projection map) of $X\times X$,
$\mathcal{K}_{0}\subseteq \mathsf{P}(\mathcal{K})$ and
$\mathcal{K}\subseteq \mathsf{H}(\mathcal{K}_{0})$. Hence,
$\mathcal{K}_{0}$ and $\mathcal{K}$ generate the same variety.
Moreover, BCK-reducts of subalgebras (homomorphic images, direct
products respectively) of algebras from $\mathsf{D}(\mathcal{K})$
are subalgebras (homomorphic images, direct products, respectively)
of the corresponding BCK-reducts. Therefore, the BCK-reduct of any
algebra in $\mathsf{V(D}(\mathcal{K}))$ is in
$\mathsf{HSP}(\mathcal{K}_{0})=
\mathsf{HSP}(\mathcal{K})=\mathsf{V}(\mathcal{K})$.

(2) Let $(X,\mu)\in \mathcal{V}_\mu $. The map $ \Phi
:a\mapsto (\mu(a),a)$ is an embedding of $(X,\mu)$ into
$D(X)$. Moreover, $\Phi (\mu(a))=(\mu(\mu  (a)),\mu
(a))=(\mu(a),\mu(a))=\mu_{X}((\mu(a),a))=\mu
_{X}(\Phi(a))$. Hence, $\Phi $ is an injective homomorphism of
state-morphism BCK-algebras, and $(X,\mu)\in
\mathsf{ISD}(\mathcal{V})$. Conversely, the BCK-reduct of any
algebra in $\mathsf{D}(\mathcal{V})$ is in $\mathcal{V}$,  and hence
the BCK-reduct of any member of $\mathsf{ISD}(\mathcal{V})$ is in
$\mathsf{IS}(\mathcal{V})= \mathcal{V}$. Hence, any member of
$\mathsf{ISD}(\mathcal{V})$ is in $\mathcal{V}_\mu$.
\end{proof}

\begin{lem}\label{main}
Let $\mathcal{K}$ be a class of BCK-algebras. Then:
\newline {\rm (1)}
$\mathsf{DH}(\mathcal{K})\subseteq \mathsf{HD}(\mathcal{K})$.
\newline {\rm (2)}
$\mathsf{DS}(\mathcal{K})\subseteq \mathsf{ISD}(\mathcal{K})$.
\newline {\rm (3)}
$\mathsf{DP}(\mathcal{K})\subseteq \mathsf{IPD}(\mathcal{K})$.
\newline {\rm (4)}
$\mathsf{V(D}(\mathcal{K}))=\mathsf{ISD}(\mathsf{V}( \mathcal{K}))$.
\end{lem}

\begin{proof}
(1) Let $D(C)\in \mathsf{DH}(\mathcal{K})$. Then there are $ X\in
\mathcal{K}$ and a BCK-homomorphism $h$ from $X$ onto $C$. Let, for all
$a,b\in X$, $h^{*}(a,b)=(h(a),h(b))$. We claim that $h^{*}$ is a 
homomorphism from  the diagonal state-morphism BCK-algebra $D(X)$ onto $D(C)$. That $ h^{*}$ is a
BCK-homomorphism is clear. We verify that $h^{*}$ is compatible with
$\mu _{X}$. We have $h^{*}(\mu
_{X}(a,b))=h^{*}(a,a)=(h(a),h(a))=\mu _{C}(h(a),h(b))=\mu
_{C}(h^{*}(a,b)).$ Finally, since $h$ is onto, given $(c,d)\in
C\times C$, there are $a,b\in X$ such that $h(a)=c$ and $h(b)=d$.
Hence, $h^{*}(a,b)=(c,d)$, $h^{*}$ is onto, and $D(C)\in
\mathsf{HD}(\mathcal{K})$.

(2) It is trivial.

(3) Let $X=\prod_{i\in I} X_{i}\in \mathsf{P}(\mathcal{K})$, where
each $X_{i}$ is in $\mathcal{K}$. We assert the map
\begin{center}
$\Phi :\bigl((a_{i}:i\in I),(b_{i}:i\in I)\bigr)\mapsto
\bigl((a_{i},b_{i}\bigr):i\in I)$
\end{center}
is an isomorphism of state-morphism BCK-algebras from $D(X)$ onto $\prod_{i\in I}D(X_{i})$. Indeed,
it is clear that $\Phi $ is a BCK-isomorphism. Moreover, denoting
the state-morphism of $\prod_{i\in I}D(X_{i})$ by $\mu ^{*}$, we
get:
\begin{eqnarray*}&\Phi \bigl(\mu _{X}\bigl((a_{i}:i\in I),(b_{i}:i\in I)\bigr)
\bigr)=\Phi
\bigl((a_{i}:i\in I),(a_{i}:i\in I)\bigr)=\\
&= \bigl((a_{i},a_{i}):i\in I\bigr)=\bigl(\mu
_{X_{i}}(a_{i},b_{i}):i\in I\bigr)=\mu ^{*}\bigl(\Phi
\bigl((a_{i}:i\in I),(b_{i}:i\in I)\bigr)\bigr),
\end{eqnarray*}
and whence $\Phi $ is an isomorphism of  state-morphism BCK-algebras.

(4) By (1), (2) and (3),
$\mathsf{DV}(\mathcal{K})=\mathsf{DHSP}(\mathcal{K} )\subseteq
\mathsf{HSPD}(\mathcal{K})=\mathsf{V(D}(\mathcal{K}))$, and hence $
\mathsf{ISDV}(\mathcal{K})\subseteq
\mathsf{ISV(D}(\mathcal{K}))=\mathsf{V(D}( \mathcal{K}))$. Conversely,
by Lemma \ref{wd}(1), $\mathsf{V(D}(\mathcal{K}))\subseteq
\mathsf{V}(\mathcal{K})_\mu$, and by Lemma \ref{wd}(2), $
\mathsf{V}(\mathcal{K})_\mu=\mathsf{ISDV}(\mathcal{K})$. This
proves the claim.
\end{proof}

\begin{thm}\label{ad:6} If a system $\mathcal K$ of BCK-algebras
generates a variety $\mathcal V$ of  BCK-algebras, then $\mathsf
D(\mathcal K)$ generates the variety $\mathcal V_\mu$ of
state-morphism BCK-algebras.
\end{thm}

\begin{proof}
By Lemma \ref{main}(4), $\mathsf{V(D}(\mathcal
K))=\mathsf{I}$\textsf{S}$\mathsf{D}(\mathsf{V}(\mathcal K))$.
Moreover, by Lemma \ref{wd}(2), $\mathsf{V}(\mathcal
K)_\mu=\mathsf{ISDV }(\mathcal K)$. Hence,
$\mathsf{V}(\mathsf{D}(\mathcal K))=\mathsf{V}(\mathcal K )_\mu$.
\end{proof}

Let $[0,1]$ be the real interval. We endow it with the BCK-structure as before: $s*_\mathbb R t = \max\{0,s-t\},$ $s,t \in [0,1]$. We denote by $[0,1]_{BCK}:=([0,1],*_\mathbb R,0)$ and it is a bounded commutative BCK-algebra. If, for bounded commutative BCK-algebras, we define a state-morphism operator $\mu$ as a homomorphism of bounded commutative BCK-algebras $\mu: X \to X$ such that $\mu \circ \mu = \mu$ and$\mu(1)=1,$ we can obtain the following result.

\begin{cor}\label{AD:co}
Let $\mathcal V$ be the variety of bounded commutative BCK-algebras, and let $\mathcal V_{BCK}$ be the variety of all bounded commutative state-morphism BCK-algebras. Then $\mathcal V_{BCK} = \mathsf V(D([0,1]_{BCK})).$
\end{cor}

\begin{proof}
We can repeat the proofs of Lemmas \ref{wd}--\ref{main} and Theorem \ref{ad:6} also for state-morphism operators on bounded commutative BCK-algebras. We have $\mathcal V_{BCK} = \mathcal V_\mu.$ By \cite{Mun}, the variety of bounded BCK-algebras is categorically equivalent to the variety of MV-algebras. Since the MV-algebra $[0,1]$ generates the variety of MV-algebras, we have that the BCK-algebra $[0,1]_{BCK}$ generates the variety of bounded commutative BCK-algebras. Then by Theorem \ref{ad:6}, we have $\mathcal V_{BCK} = \mathsf V(D([0,1]_{BCK})).$
\end{proof}

\begin{cor}
There is uncountably many subvarieties of the variety $\mathcal V_{BCK}$ of  bounded commutative BCK-algebras with a state-morphism.
\end{cor}

\begin{proof}
By \cite[Thm 7.11]{DKM}, the variety of state-morphism MV-algebras is uncountable. Because the variety of bounded commutative BCK-algebras is categorically equivalent to the variety of MV-algebras, \cite{Mun}, we have the statement in question.
\end{proof}

Now we present some analogous general results concerning quasivarieties. The proofs follow the similar ideas just used for varieties.

\begin{lem}\label{wd1}
{\rm (1)} Let $\mathcal{K}$ be a class of
BCK-algebras.
Then $\mathsf{Q_V(D}(\mathcal{K}))\subseteq
\mathsf{Q_V}(\mathcal{K})_\mu.$
\newline {\rm (2)} Let $\mathcal{V}$
be any quasivariety of BCK-algebras. Then $\mathcal{V}
_\mu=\mathsf{ISD}(\mathcal{V})$.
\end{lem}

\begin{proof}
(1) We have to prove that every BCK-reduct of a state-morphism BCK-algebra in $\mathsf{Q_V}(\mathcal{K})$ is in
$\mathsf{Q_V}(\mathcal{K})$.

Let $\mathcal{K}_{0}$ be the class of all
BCK-reducts of algebras in $\mathsf{D}(\mathcal{K})$. Let $X \in
\mathcal K,$ and let $\{0\}$ be the one-element BCK-algebra which is a subalgebra of $X.$ Then $D(X)\in \mathsf{D}(\mathcal{K}).$ The BCK-reduct of $D(X)$ is $X\times X$, and since $X$ is isomorphic to the BCK-algebra $\{0\}\times X,$ which is a subalgebra of $X \times X,$ we have $X \in \mathsf{IS}(\mathcal K_0).$ Thus $\mathcal{K}_{0}\subseteq \mathsf{P}(\mathcal{K})$ and $\mathcal{K}\subseteq \mathsf{IS}(\mathcal{K}_{0})$. By \cite[Thm 2.23, 2.25]{Burris}, we have
$\mathsf{Q_V}(\mathcal{K}_0) = \mathsf{ISP_R}(\mathcal{K}_0) \subseteq \mathsf{ISP_RP}(\mathcal{K})\subseteq \mathsf{ISIP_R}(\mathcal{K})\subseteq \mathsf{IISP_R}(\mathcal{K}) = \mathsf{ISP_R}(\mathcal{K})= \mathsf{Q_V}(\mathcal{K}).$ Similarly, $\mathsf{Q_V}(\mathcal{K}) = \mathsf{ISP_R}(\mathcal{K}) \subseteq \mathsf{ISP_RIS}(\mathcal{K}_0) \subseteq \mathsf{ISIP_RS}(\mathcal{K}_0) \subseteq \mathsf{ISISP_R}(\mathcal{K}_0) \subseteq \mathsf{IISSP_R}(\mathcal{K}_0) = \mathsf{ISP_R}(\mathcal{K}_0) = \mathsf{Q_V}(\mathcal{K}_0).$  Hence, $\mathcal K$ and $\mathcal K_0$ generates the same quasivariety.

Moreover, BCK-reducts of subalgebras (isomorphic  images, reduced
products, respectively) of algebras from $\mathsf{D}(\mathcal{K})$
are subalgebras (isomorphic images, reduced products, respectively)
of the corresponding BCK-reducts. Therefore, the BCK-reduct of any
algebra in $\mathsf{Q_V(D}(\mathcal{K}))$ is in
$\mathsf{Q_V}(\mathcal K_0)=
\mathsf{Q_V}(\mathcal{K})=\mathsf{Q_V}(\mathcal{K}),$ which proves (1).

(2) Let $(X,\mu)\in \mathcal{V}_\mu $. The map $ \Phi
:a\mapsto (\mu(a),a)$ is an embedding of $(X,\mu)$ into
$D(X)$. Moreover, $\Phi (\mu(a))=(\mu(\mu  (a)),\mu
(a))=(\mu(a),\mu  (a))=\mu_{X}((\mu  (a),a))=\mu
_{X}(\Phi(a))$. Hence, $\Phi $ is an injective homomorphism of
state-morphism BCK-algebras, and $(X,\mu)\in
\mathsf{ISD}(\mathcal{V})$. Conversely, let $X \in \mathcal V.$ Then the BCK-reduct of $D(X)$ is $X \times X,$ and $X\times X$ is isomorphic with the reduced product $(X\times X)/F,$ where $F$ is the one-element filter $F =\{1,2\}$ of the set $I=\{1,2\}.$ Hence, $X\times X$ is in $\mathcal V,$ and the BCK-reduct of any
algebra in $\mathsf{D}(\mathcal{V})$ is in $\mathcal{V},$  whence
the BCK-reduct of any member of $\mathsf{ISD}(\mathcal{V})$ is in
$\mathsf{IS}(\mathcal{V})= \mathcal{V}$. Therefore, any member of
$\mathsf{ISD}(\mathcal{V})$ is in $\mathcal{V}_\mu$.
\end{proof}

\begin{lem}\label{mainqv}
Let $\mathcal{K}$ be a class of BCK-algebras. Then:
\newline {\rm (1)}
$\mathsf{DI}(\mathcal{K})\subseteq \mathsf{ID}(\mathcal{K})$.
\newline {\rm (2)}
$\mathsf{DS}(\mathcal{K})\subseteq \mathsf{ISD}(\mathcal{K})$.
\newline {\rm (3)}
$\mathsf{DP_R}(\mathcal{K})\subseteq \mathsf{IP_RD}(\mathcal{K})$.
\newline {\rm (4)}
$\mathsf{Q_V(D}(\mathcal{K}))=\mathsf{ISD}(\mathsf{Q_V}( \mathcal{K}))$.
\end{lem}

\begin{proof}
(1) Let $D(C)\in \mathsf{DI}(\mathcal{K})$. Then there are $ X\in
\mathcal{K}$ and an isomorphism $h$ from $X$ onto $C$. Let, for all
$a,b\in X$, $h^{*}(a,b)=(h(a),h(b))$. We claim that $h^{*}$ is an
isomorphism from $D(X)$ onto $D(C)$. That $ h^{*}$ is an
isomorphism of BCK-algebras is clear. We verify that $h^{*}$ is compatible with
$\mu _{X}$. We have $h^{*}(\mu
_{X}(a,b))=h^{*}(a,a)=(h(a),h(a))=\mu _{C}(h(a),h(b))=\mu
_{C}(h^{*}(a,b)).$ Finally, since $h$ is onto, given $(c,d)\in
C\times C$, there are $a,b\in X$ such that $h(a)=c$ and $h(b)=d$.
Hence, $h^{*}(a,b)=(c,d)$, $h^{*}$ is onto, and $D(C)\in
\mathsf{ID}(\mathcal{K})$.

(2) It is trivial.

(3) Let $X=\prod_{i\in I} X_{i}/F\in \mathsf{P_R}(\mathcal{K})$, where
each $X_{i}$ is in $\mathcal{K}$, and $F$ is a filter over $I.$ We claim the map
\begin{center}
$\Phi :\bigl((a_{i}:i\in I)/F,(b_{i}:i\in I)/F\bigr)\mapsto
\bigl((a_{i},b_{i}\bigr):i\in I)/F$
\end{center}
is an isomorphism from $D(X)$ onto $\prod_{i\in I}D(X_{i})/F$. Indeed,
it is clear that $\Phi $ is a BCK-isomorphism: let $\bigl((a_{i},b_{i}\bigr):i\in I)/F = \bigl((a_{i}',b_{i}'\bigr):i\in I)/F.$ Then $\llbracket a_i=a_i'\rrbracket \cap \llbracket b_i = b_i' \rrbracket= \llbracket(a_i,b_i)=(a_i',b_i') \rrbracket \in F,$ so that $\llbracket a_i=a_i'\rrbracket,\llbracket b_i=b_i'\rrbracket\in F$ and hence $\bigl((a_{i},b_{i}\bigr):i\in I)/F =\bigl((a_{i}',b_{i}\bigr):i\in I)/F.$ Moreover, denoting
the state-morphism of $\prod_{i\in I}D(X_{i})$ by $\mu ^{*}$, we
get:
\begin{eqnarray*}&\Phi \bigl(\mu _{X}\bigl((a_{i}:i\in I)/F,(b_{i}:i\in I)/F\bigr)
\bigr)=\Phi
\bigl((a_{i}:i\in I)/F,(a_{i}:i\in I)\bigr)/F=\\
&= \bigl((a_{i},a_{i}):i\in I\bigr)/F=\bigl(\mu
_{X_{i}}(a_{i},b_{i}):i\in I\bigr)=\mu ^{*}\bigl(\Phi
\bigl((a_{i}:i\in I)/F,(b_{i}:i\in I)/F\bigr)\bigr),
\end{eqnarray*}
and hence, $\Phi $ is a state-morphism isomorphism.

(4) By (1), (2) and (3),
$\mathsf{DQ_V}(\mathcal{K})=\mathsf{DISP_R}(\mathcal{K} )\subseteq
\mathsf{IISIP_RD}(\mathcal{K})\subseteq \mathsf{IIISP_RD}(\mathcal{K}) = \mathsf{ISP_RD}(\mathcal{K}) =\mathsf{Q_V(D}(\mathcal{K}))$, and hence, $
\mathsf{ISDQ_V}(\mathcal{K})\subseteq
\mathsf{ISQ_V(D}(\mathcal{K}))=\mathsf{Q_V(D}( \mathcal{K}))$. Conversely,
by Lemma \ref{wd1}(1), $\mathsf{Q_V(D}(\mathcal{K}))\subseteq
\mathsf{Q_V}(\mathcal{K})_\mu$, and by Lemma \ref{wd1}(2), $
\mathsf{Q_V}(\mathcal{K})_\mu=\mathsf{ISDQ_V}(\mathcal{K})$. This
proves the claim.
\end{proof}

Finally, we present the main result of the section about generators of quasivarieties of state-morphism BCK-algebras which is an analogue of Theorem \ref{ad:6}.

\begin{thm}\label{ad:7} If a system $\mathcal K$ of BCK-algebras
generates a quasivariety $\mathcal V$ of BCK-algebras, then $\mathsf
D(\mathcal K)$ generates the quasivariety $\mathcal V_\mu$ of
state-morphism BCK-algebras.
\end{thm}

\begin{proof}
By Lemma \ref{mainqv}(4), $\mathsf{Q_V(D}(\mathcal
K))=\mathsf{ISD}(\mathsf{Q_V}(\mathcal K))$.
Moreover, by Lemma \ref{wd1}(2), $\mathsf{Q_V}(\mathcal
K)_\mu=\mathsf{ISD(Q_V }(\mathcal K))$. Hence,
$\mathsf{Q_V}(\mathsf{D}(\mathcal K))=\mathsf{Q_V}(\mathcal K )_\mu$.
\end{proof}

Since the interval $[0,1]$ generates the class $\mathcal{MV}$ of MV-algebras as both a variety and a quasivariety, due to the categorical equivalence of MV-algebras and bounded commutative BCK-algebras, \cite{Mun},
by Theorem \ref{ad:7} and Corollary \ref{AD:co}, we have the following corollary.

\begin{cor}\label{co:ad7}
If $[0,1]_{BCK}=([0,1],*_\mathbb R,0)$ is the bounded commutative BCK-algebra of the real interval $[0,1]$, then $D([0,1]_{BCK})$ generates both as the variety and as the quasivariety of state-morphism BCK-algebras whose BCK-reduct is a bounded commutative BCK-algebra. In other words, $  \mathsf {V}(D([0,1]_{BCK}))= \mathcal V_{BCK} = \mathsf {Q_V}(D([0,1]_{BCK})).$
\end{cor}

Finally, we formulate two open problems.

\vspace{2mm}\noindent
{\bf Problem 1.} {\it Describe some interesting generators of the quasivariety of state BCK-algebras.}
\vspace{2mm}

We note that we do not know yet any interesting generator for the variety of state MV-algebras.

(2) If $X$ is a subdirectly irreducible BCK-algebra, then the diagonal state-morphism BCK-algebra $D(X)$ is subdirectly irreducible. Similarly, if $X$ is linearly ordered and subdirectly irreducible,  then $(X,{\rm Id}_X)$ is subdirectly irreducible. If $X$ is an MV-algebra, the third category of subdirectly irreducible state-morphism MV-alegbra $(X,\mu)$ is the case when $X$ has a unique maximal ideal. Inspired by that, we formulate the second open problem:

\vspace{2mm}\noindent
{\bf Problem 2.} {\it Characterize (bounded) subdirectly irreducible state-morphism BCK-algebras as it was done in \cite{DD1, DDL1, DKM}.}



\begin{thebibliography}{99}


\bibitem{AgMo}
M. Aguiar, W. Moreira,
{\it Combinatorics of the free Baxter algebra,}
Electron. J. Combin. {\bf 13} (2006),  R 17, 38 pp.

\bibitem{BZ}
R. A. Borzooei, O. Zahiri, {\it Prime ideals in BCI-algebras and BCK-algebras}, Ann.  Univer.  Craiova {\bf 32} (2012), 299--309.

\bibitem{BoDv}
M. Botur, A. Dvure\v{c}enskij, {\it State-morphism algebras - general approach,} Fuzzy Sets and Systems {\bf 218} (2013), 90--102.  DOI: http://dx.doi.org/10.1016/j.fss.2012.08.013


\bibitem{Burris}
S. Burris, H. P. Sankappanavar, {\it ``A Course in Universal Algebra"}, Springer-Verlag, New York, 1981.

\bibitem{DD1}
M. Botur, A. Dvure\v{c}enskij, {\it State-morphism algebras-General approach}, Fuzzy Sets and Systems {\bf 218} (2013), 90--102.

\bibitem{MV}
C. C. Chang, {\it Algebraic analysis of many-valued logics},  Trans. Amer. Math. Soc. {\bf 88} (1958), 467--490.

\bibitem{CD}
L. C. Ciungu, A. Dvure\v{c}enskij,  {\it Measures, states and de Finetti maps on pseudo-BCK algebras},
Fuzzy Sets and Systems {\bf 161} (2010), 2870--2896. DOI:10.1016/j.fss.2010.03.017

\bibitem{DD1}
A. Di Nola, A. Dvure\v{c}enskij,
{\it State-morphism MV-algebras}, Ann. Pure Appl. Logic. {\bf 161} (2009), 161--173.

\bibitem{DD2}
A. Di Nola, A. Dvure\v{c}enskij, {\it On some classes of state-morphism MV-algebras}, Math. Slovaca {\bf 59} (2009), 517--534.

\bibitem{DDL}
A. Di Nola, A. Dvure\v{c}enskij, A. Lettieri, {\it On varieties of MV-algebras with internal states}, Internat. J. Approx. Reason. {\bf 51} (2010), 680--694.
DOI: 10.1016/j.ijar.2010.01.017

\bibitem{DDL1} A. Di Nola, A. Dvure\v{c}enskij, A. Lettieri,
{\it Erratum   ``State-morphism MV-algebras" [Ann. Pure Appl. Logic
161 (2009) 161-173],} Ann. Pure Appl. Logic {\bf 161} (2010),
1605--1607. DOI 10.1016/j.apal.2010.06.004

\bibitem{Dvu1}
A. Dvure\v{c}enskij, {\it Measures and states on BCK-algebras},  Sem. Mat. Fis. Univ. Modena {\bf 47} (1999), 511--528.

\bibitem{DKM} A. Dvure\v{c}enskij, T. Kowalski, F. Montagna, {\it
State morphism MV-algebras,} Inter. J. Approx. Reasoning {\bf 52}
(2011), 1215--1228. DOI: 10.1016/j.ijar.2011.07.003

\bibitem{DRS} A. Dvure\v censkij,  J. Rach\r{u}nek, D.
\v{S}alounov\'a,
{\it State operators on generalizations of fuzzy structures,}  Fuzzy
Sets and Systems {\bf 187} (2012), 58--76. DOI:     10.1016/j.fss.2011.05.023


\bibitem{Flamino1}
T. Flamino, F. Montagna, {\it An algebraic approach to states on MV-algebras}, In: M. \v{S}t\v{e}pni\v{c}ka, V. Novak, U. Bodenhofer (Eds.), Proceedings of EUSFLAT07 {\bf 2} (2007), pp. 201--206.

\bibitem{Flamino2}
T. Flaminio and F. Montagna, {\it MV-algebras with internal states and probabilistic fuzzy
logics}, Inter. J. Approx. Reasoning, {\bf 50} (2009), 138--152.

\bibitem{Haj} P. H\'{a}jek, \emph{``Metamathematics of Fuzzy
Logic"}, Kluwer Academic Publishers, Dordrecht 1998.

\bibitem{2}
Y. Imai, K. Iseki, {\it On axiom system of propositional calculi, XIV},  Japan Acad. {\bf 42} (1966), 19--22.

\bibitem{3}
K. Iseki, {\it An algebra related with a propositional calculus},  Japan Acad. {\bf 42} (1966), 26--29.

\bibitem{12}
T. Kroupa, {\it Every state on semisimple MV-algebra is integral}, Fuzzy Sets and Systems {\bf 157} (2006), 2771--2782.

\bibitem{13}
J. K\"{u}hr, D. Mundici, {\it De Finetti theorem and Borel states in $[0,1]$-valued algebraic logic},
Inter. J. Approx. Reasoning, {\bf 46} (2007), 605--616.

\bibitem{BCK}
J. Meng, Y. B. Jun, {\it ``BCK-algebras"}, Kyung Moon Sa Co., Seoul, 1994.

\bibitem{Mun}
D. Mundici, {\it MV-algebras are categorically equivalent to bounded commutative BCK-algebras}, Math. Japonica,
{\bf 31}, (1986), 889--894.

\bibitem{Mun3}
D. Mundici, {\it Averaging the truth-value in \L ukasiewicz logic},  Studia Logica {\bf 55} (1995), 113--127.

\bibitem{Mun2}
D. Mundici, {\it Tensor product and the Loomis-Sikorski theorem for MV-algebras},  Advance  Appl. Math. {\bf 22} (1999), 227--248.

\bibitem{Pan}
G. Panti, {\it Invariant measures in free
MV-algebras}, Comm. Algebra  {\bf 36} (2008), 2849--2861.

\bibitem{Sch}
O. M. Schn\"urer, {\it Homotopy categories and idempotent
completeness, weight structures and weight complex functors,}
arXiv:1107.1227v1

\bibitem{Vic}
S. Vickers, {\it  Entailment systems for stably locally compact  locales,}
Theoret. Comput. Sci. {\bf 316} (2004), 259--296.
DOI:10.1016/j.tcs.2004.01.033

\bibitem{chang}
H. Yisheng, {\it ``BCI-algebra"}, Science Press, China, 2006.
\end{thebibliography}
\end{document}